\documentclass[a4paper,oneside,10pt,oldfontcommands,reqno]{amsart}
\usepackage[a4paper, total={426pt, 674pt}]{geometry}

\newcommand{\langue}{anglais}	

\usepackage{ifthen}
\ifthenelse{\equal{\langue}{francais}}{
	\usepackage[french]{babel}}
	{\ifthenelse{\equal{\langue}{anglais}}{\usepackage[english]{babel}}{}}
\usepackage[T1]{fontenc}
\usepackage[utf8]{inputenc}
\usepackage[babel,german=guillemets]{csquotes}
\usepackage{eso-pic}

\usepackage[backend=biber, style=alphabetic, sorting=nyt, giveninits=true, maxnames=99, maxcitenames=99, doi=false, isbn=false, url=false, eprint=false]{biblatex} 
\DeclareNameAlias{author}{last-first}
\usepackage{lmodern} 
\usepackage{enumitem}
\usepackage{ifthen}
\ifthenelse{\equal{\langue}{francais}}{
	\frenchbsetup{StandardLists=true}}
	{}

\usepackage{graphicx} 
\usepackage{subcaption}

\usepackage{amsmath}
\usepackage{amssymb}
\usepackage{amsthm}
\usepackage{amsfonts}
\usepackage{bbold}
\usepackage{mathtools}
\usepackage{tikz-cd}
\usepackage{hyperref}
\usepackage{multirow}
\usepackage[normalem]{ulem}
\usepackage{cases}

\ifthenelse{\equal{\langue}{francais}}{
	\newcommand{\theoremenom}{Théorème}
	\newcommand{\propositionnom}{Proposition}
	\newcommand{\lemmenom}{Lemme}
	\newcommand{\corollairenom}{Corollaire}
	\newcommand{\definitionnom}{Définition}
	\newcommand{\remarquenom}{Remarque}
	\newcommand{\exemplenom}{Exemple}
	\newcommand{\conjecturenom}{Conjecture}
}{\ifthenelse{\equal{\langue}{anglais}}{
	\newcommand{\theoremenom}{Theorem}
	\newcommand{\propositionnom}{Proposition}
	\newcommand{\lemmenom}{Lemma}
	\newcommand{\corollairenom}{Corollary}
	\newcommand{\definitionnom}{Definition}
	\newcommand{\remarquenom}{Remark}
	\newcommand{\exemplenom}{Example}
	\newcommand{\conjecturenom}{Conjecture}
}{}}
	
\newtheorem{theoreme}{\theoremenom}[section]
\newtheorem{proposition}[theoreme]{\propositionnom}
\newtheorem{lemme}[theoreme]{\lemmenom}
\newtheorem{corollaire}[theoreme]{\corollairenom}

\newtheorem{remarque}[theoreme]{\remarquenom}

\usepackage{thmtools}
\usepackage{thm-restate}

\makeatletter
\def\cleartheorem#1{%
    \expandafter\let\csname#1\endcsname\relax
    \expandafter\let\csname c@#1\endcsname\relax
}
\makeatother
\newcommand{\compteurThm}{1}
\newcounter{annexe}

\newtheorem{thmx}{Theorem}

\newcommand{\C}{\mathbb{C}}
\newcommand{\R}{\mathbb{R}}
\newcommand{\N}{\mathbb{N}}

\newcommand{\dint}{\ds\int}
\newcommand{\ds}{\displaystyle}
\newcommand{\dsum}{\ds\sum}
\newcommand{\eqskip}{ \vspace*{2mm}\\ }
\newcommand{\fr}[2]{\frac{\ds #1}{\ds #2}}

\renewcommand{\(}{\left(}

\newcommand{\ccc}{\color{red}}

\newcommand{\op}[1]{S_{#1}}
\newcommand{\mygeq}{\trianglerighteq}
\newcommand{\myleq}{\trianglelefteq}

\newenvironment{acknowledgements}{%
  \begin{abstract}
}{%
  \end{abstract}
}

\begin{document}

\pagestyle{empty} 


\title{Extremal problems for clamped plates under tension}

\author{Pedro Freitas}\email{pedrodefreitas@tecnico.ulisboa.pt}

\author{Roméo Leylekian}\email{romeo.leylekian@tecnico.ulisboa.pt}

\address{Grupo de F\'{i}sica Matem\'{a}tica, Instituto Superior Técnico, Universidade de Lisboa, Av. Rovisco Pais, 1049-001 Lisboa, Portugal}
\maketitle


\begin{abstract}
We address extremum problems for spectral quantities associated with operators of the form $\Delta^2-\tau\Delta$ with Dirichlet boundary conditions, for non-negative values of $\tau$. The focus is on two shape optimisation problems: minimising the first eigenvalue; and maximising the torsional rigidity, both under volume constraint. We establish, on the one hand, a Szeg\H{o}-type inequality, that is, we show that among all domains
having a first eigenfunction of fixed sign the ball minimises the corresponding first eigenvalue; on the other hand a Saint--Venant-type inequality,
namely, a sharp upper bound on the torsional rigidity, again achieved by the ball. We further present other properties related to these operators and express the optimality 
condition associated with the minimisation of the first eigenvalue.

\end{abstract}



\pagestyle{plain} 


\section{Introduction}

This paper is devoted to the study of extremal shapes associated with a family of operators used as a model for solid plates.
These operators may be written as $\op{\tau} = \Delta^2-\tau\Delta,$ where $\tau$ is a real number, and they are complemented
with standard Dirichlet boundary conditions. Operators such as $\op{\tau}$ are part of a class of what might be called 
\emph{mixed-order operators}, as they involve differential operators of different orders. As a consequence, they are not invariant
under scaling, for instance. Furthermore, the behaviour of $\op{\tau}$ depends strongly on the sign of the parameter $\tau$, both from a mathematical and a physical perspective. Throughout, we will focus on the case where $\tau$ is non-negative, which corresponds to
a model for clamped plates subject to a uniform lateral tension proportional to $\tau$ \cite{weinstein-chien,payne,kawohl-levine-velte}.

We will be interested in different spectral quantities associated with the above operators, and especially in their dependence on the domain and on the real parameter $\tau$. First, we will consider the eigenvalues of $S_\tau$ defined, on an open subset $\Omega$ in $d-$dimensional Euclidean space $\R^d$, as the (real) numbers $\Gamma$ for which the problem
\begin{equation}\label{eq:pb vp}
\left\{
\begin{array}{rcll}
\Delta^2 u-\tau\Delta u & = & \Gamma u & \text{in }\Omega,\\
u & = & 0 &\text{on }\partial\Omega,\\
\partial_n u & = & 0 &\text{on }\partial\Omega,
\end{array}
\right.
\end{equation}
admits a nontrivial solution $u:\Omega\to\R$. We recall that $\Delta$ and $\Delta^2$ denote the Laplacian and the bilaplacian,
respectively, with the latter being defined by applying twice the former. When $\Omega$ has finite volume there exists a non-decreasing sequence of eigenvalues
going to infinity, each with finite multiplicity. We will focus on the first of those eigenvalues that we will denote by $\Gamma(\Omega,\tau)$, in order
to emphasise its dependence on both the domain $\Omega$ and the parameter $\tau$. Our goal will be to identify the domains $\Omega$ for which
$\Gamma(\Omega,\tau)$ is minimal, among domains of equal volume. This corresponds to solving the following minimisation problem:
\begin{equation}\label{eq:pb min}
\min\{\Gamma(\Omega,\tau):\Omega\text{ open subset of }\R^d\text{, }|\Omega|=c\},
\end{equation}
where $|\Omega|$ is the volume of $\Omega$ and $c$ is an arbitrary (positive) parameter. It is expected that the Euclidean ball is a solution to this problem or, in other words, that the first eigenvalue of~\eqref{eq:pb vp} satisfies a
Faber--Krahn-type inequality. Our first main result consists in a partial answer to this question. More precisely, we will prove that a Faber--Krahn inequality holds for all positive values of $\tau$ for domains whose first eigenfunction does not change sign, thus extending the result obtained by
Szeg\H{o} for vanishing $\tau$~\cite{szego}.

\begin{thmx}[Szeg\H{o}-type result]\label{thmsz}
Let $\tau\geq0$ and $\Omega\subseteq\R^d$ be an open set of finite volume over which an eigenfunction corresponding to the first eigenvalue of $S_\tau$
does not change sign. Then,
$$
\Gamma(\Omega,\tau)\geq\Gamma(B,\tau),
$$
where $B$ is a ball of same volume as $\Omega$.
\end{thmx}

As far as we are aware, this result is the first of its type for $S_\tau$ when $\tau$ is positive. However, as
usual with fourth-order operators, a major issue is the lack of positivity preservation, which makes the sign assumption on the first eigenfunction
very restrictive. We point out that on the unit ball, De Coster, Nicaise and Troestler have obtained a complete description of the eigenvalues and eigenfunctions of $S_\tau$ in \cite{coster-nicaise-troestler15,coster-nicaise-troestler17}. In particular, they proved that the first eigenfunctions are of fixed sign as long as the first eigenvalue is greater than $-j_{\nu,1}^2j_{\nu,2}^2$, whereas they are sign changing when the first eigenvalue is strictly below this threshold. Here and after, $\nu=d/2-1$ and $(j_{\nu,i})_{i\in\N^*}$ denote the positive zeros of the Bessel function of first kind and of order $\nu$. Since the first eigenvalue is positive when $\tau$ is non-negative, this shows that balls do satisfy the assumption of Theorem~\ref{thmsz}. Anyway, just like in Szeg\H{o}'s original result, we leave open the case of domains for which the first eigenfunction changes sign.

The other problem that we address in this paper is related to the torsion function, defined as the solution $w:\Omega\to\R$ of the following equation:
\begin{equation}\label{torsion-eq}
\left\{
\begin{array}{rcll}
\Delta^2 w-\tau\Delta w & = & 1 &\text{in }\Omega,\\
w & = & 0 &\text{on }\partial \Omega,\\
\partial_n w & = & 0 &\text{on }\partial \Omega.
\end{array}
\right.
\end{equation}
In this case the quantity under study is the torsional rigidity, also called sometimes compliance or mean deflection, and defined by
\begin{equation*}
 T(\Omega,\tau) = \dint_{\Omega} w.
\end{equation*}
In general, the torsion function may also change sign, which leads to similar issues
as above (see however Theorem~\ref{thm:torsion boule} for the case of balls). Despite this difficulty, here we prove the full isoperimetric
inequality for the torsional rigidity, valid in any dimension. This generalises the well-known Saint--Venant inequality for the
Laplacian, proved by P\'olya~\cite{saint-venant,polya}, and its counterpart for the bilaplacian, recently obtained by Ashbaugh, Bucur, Laugesen, and Leylekian~\cite{ashbaugh-bucur-laugesen-leylekian}.
In fact, our result may be seen as a bridge, making the link between those two results.

\begin{thmx}[Saint--Venant-type inequality]\label{thmsv}
Let $\Omega$ be an open subset of $\R^d$ of finite volume. Then, for any non-negative value of $\tau$ we have
$$
T(\Omega,\tau)\leq T(B,\tau),
$$
where $B$ is a ball of same volume as $\Omega$.
\end{thmx}

\begin{remarque}
The case of equality in Theorems~\ref{thmsz} and~\ref{thmsv} involves the notion of capacity associated to the norm $H^2(\R^d)$, or $H^2$-capacity (we refer to \cite[Section 3.8]{henrot-pierre} for an introduction, and to \cite{adams-hedberg} for a comprehensive presentation of higher order capacities). Indeed, it is known that removing a set of zero $H^2$-capacity from a domain $\Omega$ does not change either the first eigenvalue or the torsional rigidity over $\Omega$ (for any $\tau\in\R$). As a result, it is expected that equality holds in the above theorems if and only if $\Omega$ is a ball up to a set of zero $H^2$-capacity. Usually, this claim may be proved provided that the equality case in Talenti's comparison principle is clarified (see \cite{ashbaugh-bucur-laugesen-leylekian}). Since the version of Talenti's comparison principle that is used here (Theorem~\ref{thm:talenti}) does not contain any statement on the case of equality, we leave this question open for Theorem~\ref{thmsz} and Theorem~\ref{thmsv}.
\end{remarque}

\subsection*{Motivations and state-of-the-art}

Before going further, let us explain in more detail our interest in this class of mixed-order operators. These operators are of particular significance 
insofar as they are one of the simplest ways to interpolate between the Laplacian and the bilaplacian. Indeed, while $\op{0}=\Delta^{2}$ together
with Dirichlet boundary conditions, $\op{\tau}$ will converge (at least formally) to the Laplacian, also with Dirichlet boundary
conditions, when $\tau$ goes to $+\infty$.

The relevant point here is that the Laplacian and the bilaplacian exhibit very different features. In particular, we mention the preservation of positivity 
and the existence of a maximum principle, two properties that are satisfied by the former operator, but not the latter in general. Moreover, even  when a property holds for both operators, the corresponding proofs are, in general, quite different. One striking example is the 
Faber--Krahn inequality, stating that the ball is a minimser for the first eigenvalue of the Dirichlet Laplacian subject to a volume 
restriction~\cite{faber,krahn}. While this is known to hold in  any dimension, the analogous problem for the bilaplacian has been solved only in dimension
two by Nadirashvili and in dimension three by Ashbaugh and Benguria~\cite{nadirashvili,ashbaugh-benguria}. In higher dimensions, only partial results are available, such as Szeg\H{o}'s 
result~\cite{szego} (see also~\cite{ashbaugh-laugesen,leylekian1,leylekian2}). Consequently, it is of interest to investigate intermediate situations.

In this respect, operators of mixed-order such as $\op{\tau}$ constitute an appropriate avenue, with the case of a non-negative
tension parameter $\tau$ being of particular significance, for it gives rise to intermediate operators between the Laplacian and the bilaplacian. Furthermore,
it might be expected that, as $\tau$ becomes larger, the behaviour of $\op{\tau}$ approaches that of the Laplacian \cite{kawohl-levine-velte,buoso-kennedy}. To some extent, this is indeed what
happens for classical properties such as the maximum principle. Indeed, for a given domain and a given positive function $f$, the solution of the equation $\Delta^2u-
\tau\Delta u=f$, with Dirichlet boundary conditions, is eventually positive when $\tau$ goes to infinity, with a threshold depending only on the domain in dimensions $2$ and $3$~\cite{cassani-tarsia}, but depending on the domain and also possibly on the function $f$ in higher dimensions~\cite{eichmann-schatzle}. Let us also point out that for radial functions on the unit ball, positivity preservation holds for all values of $\tau\geq0$, as shown in \cite{laurençot-walker}, and it is actually possible to extend this result to $\tau> -j_{\nu}^{2}$ using similar arguments.

Suprisingly, this expected improvement as $\tau$ increases has not manifested itself so far when it comes to the determination of optimal shapes for spectral quantities. Instead, the only known results are actually for negative values of $\tau$. Physically speaking, this range
corresponds to a situation where the plate is subject to a lateral compression, rather than tension. In this setting, Ashbaugh, Benguria and Mahadevan have developed some work in this direction. As was announced in~\cite{ashbaugh-benguria-mahadevan1,ashbaugh-benguria-mahadevan2}, they obtained that the first eigenvalue under a volume constraint is minimised by the ball, for negative $\tau$ close enough to $0$, in dimensions $2$ and $3$. Similarly, in the case of Saint--Venant-type results for the torsion function $T(\Omega,\tau)$ defined above, the only known results are in the recent paper~\cite{ashbaugh-bucur-laugesen-leylekian} where, in particular, the torsional rigidity is proved to be maximised by the ball in dimension $2$, again for negative values of the  parameter $\tau$ close enough to $0$, and in any dimension for vanishing $\tau$ (see
also~\cite[Theorem 5.18]{leylekian-these}).

Overall, all the known results regarding shape optimisation have been, so far, obtained for non-positive values of the parameter $\tau$. The reason comes from the fact that
they all rely on a symmetrisation procedure introduced by Talenti \cite{talenti76}, allowing to build from a given function a radially symmetric competitor with improved
properties. One of these properties is the fact that the gradient of the competitor has greater $L^2$ norm than that of the initial function.
Somewhat paradoxically, this works appropriately for the variational characterization of the quantities involved for non-positive values of $\tau$,
but not for positive values of this parameter. Indeed, when $\tau$ is non-positive, the strategy initiated in~\cite{talenti81} and developed
in~\cite{nadirashvili,ashbaugh-benguria,ashbaugh-laugesen} can be implemented to give rise to a \emph{two-ball} problem, whose solution yields the results 
mentioned above. However, when $\tau$ is positive the growth of the gradient norm is no longer compatible with the variational characterization, and a new 
symmetrisation procedure has to be developed. Let us however mention the works of Chasman~\cite{chasman-these,chasman11,chasman16} on the free plate under 
tension. Yet for those free boundary conditions, the shape optimisation problem associated with the first nontrivial eigenvalue is a maximisation problem, 
hence the setting and the techniques are completely different. Finally, a Steklov type problem related to plates under tension was considered by Buoso and Provenzano in \cite{buoso-provenzano}, and a sharp quantitative upper bound was obtained on the first nontrivial eigenvalue.

\subsection*{Sketch of the proofs}

Without getting into the details, let us give a first insight into our proofs of Theorem \ref{thmsz} and Theorem \ref{thmsv}.

\subsubsection*{Szeg\H{o}-type result}

The first eigenvalue of $S_{\tau}$ admits the variational characterization (see Section \ref{sec:preliminaires})
$$
\Gamma(\Omega,\tau)=\min_{\substack{u\in H_0^2(\Omega),\\u\neq0}}\fr{\dint_\Omega(\Delta u)^2+\tau\int_\Omega|\nabla u|^2}{\dint_\Omega u^2}.
$$
This shows that, if we want to decrease $\Gamma(\Omega,\tau)$, we should not increase the norm of the gradient. Therefore the traditional idea of Talenti
of symmetrising the Laplacian of a test function is not appropriate here (see point {\sc v} of \cite[Theorem~1]{talenti76}). To circumvent this issue, our main idea is to factorise the mixed-order operator by remarking that
$$
\Delta^2-\tau\Delta u=\Gamma u\quad\Leftrightarrow\quad \left(\Delta -\frac{\tau}{2}\right)^{2}u=\left(\Gamma+\frac{\tau^2}{4}\right)u.
$$
Therefore, up to a translation of the spectrum, the operator $S_{\tau}$ is equivalent to the operator $(\Delta-\tau/2)^2$, whose first eigenvalue
admits the variational characterization
$$
\min_{\substack{u\in H_0^2(\Omega),\\u\neq0}}\fr{\dint_\Omega\left(\Delta u-\frac{\tau}{2}u\right)^2}{\dint_\Omega u^2}.
$$
Consequently, one observes that, instead of symmetrising the Laplacian of a test function $u$, it makes sense to symmetrise $\Delta u-\tau/2 u$, and to define the
competitor $v\in H_0^1$ such that
$$
-\Delta v+\frac{\tau}{2} v =\left(-\Delta u+\frac{\tau}{2} u\right)^*,
$$
where $^*$ denotes the standard Schwarz symmetrisation for signed functions~\cite{kesavan}. This operation does not increase the previous quotient, at least when $u$ is of fixed sign (see Theorem \ref{thm:talenti}). However, a major issue concerns the boundary conditions for $v$. Indeed, as we shall see, we can only guarantee that $v$ satisfies the \emph{incomplete boundary condition} $v\mygeq 0$ and $\partial_nv\geq0$ over $\partial\Omega^*$, where $v\mygeq0$ just means that the integral of $v$ over balls centered at zero is non-negative (see \eqref{eq:concentration}). Consequently, we end up with the following lower bound, involving this incomplete boundary condition:
$$
\Gamma(\Omega,\tau)+\frac{\tau^2}{4}\geq\min_{\substack{v\in H_0^1\cap H^2(\Omega^*),\\v\mygeq0,\partial_nv\geq0}}\fr{\dint_\Omega\left(\Delta v-\frac{\tau}{2}v\right)^2}{\dint_\Omega v^2}.
$$
Observe that if we had $v\geq0$ instead of $v\mygeq0$, we would immediately conclude that $v\in H_0^2(\Omega^*)$, and the above quotient would
become equivalent to the variational characterization of $\Gamma(\Omega^*,\tau)$ (modulo the translation factor $\tau^2/4$). After a careful analysis, we were able to complete the boundary condition, and to prove that any minimiser of the quotient is actually in $H_0^2(\Omega^*)$ (see Proposition \ref{prop:pb 1 boule}). Basically, the idea is that if not saturated, the constraint on the normal derivative of $v$ would become void, and hence it would be equivalent to minimising over the whole space $H_0^1\cap H^2(\Omega^*)$. This would correspond to the variational characterization of the first eigenvalue of the operator $(\Delta-\tau/2)^2$ under Navier boundary conditions $v=\Delta v=0$ on $\partial\Omega^*$. But the associated eigenfunction does not satisfy the initial incomplete boundary condition, which is a contradiction. Overall, this yields our Szeg\H{o}-type result.

In the case where $u$ changes sign, the situation is more complicated. Indeed, we have to split $u$ into positive and negative parts. But, as shown above, the
symmetrisation procedure that we employ preserves the boundary condition only incompletely for each of the positive and the negative parts. The loss
of information is then too strong and does not allow us to finish the proof -- see Remark \ref{rmq:echec FK avec fct propre signee} for more details.

\subsubsection*{Saint--Venant-type inequality}
As usual, the situation appears to be more favourable in the case of the torsional rigidity, as for this functional more variational characterizations are available. One of them is:
$$
-T(\Omega,\tau)=\min_{u\in H_0^2(\Omega)}\int_\Omega(\Delta u)^2+\tau\int_\Omega|\nabla u|^2-2\int_\Omega u.
$$
As before, we may factorise it in order to display the term $\Delta u-\tau/2u$, which is compatible with our new symmetrisation procedure:
$$
\int_\Omega(\Delta u)^2+\tau\int_\Omega|\nabla u|^2-2\int_\Omega u=\int_\Omega(\Delta u-\tau/2u)^2-\tau^2/4\int_\Omega u^2-2\int_\Omega u.
$$
Let us address immediately the general case where $u$ changes sign, and split $u$ into positive and negative parts. If we applied our symmetrisation method to both $u_+$
and $u_-$, we would face the issue mentioned above of loss of information regarding the boundary condition. Therefore, inspired by the asymmetric treatment proposed in \cite{ashbaugh-bucur-laugesen-leylekian}, our idea was to process $u_+$ and $u_-$ in different ways.
Indeed, while for $u_+$ we symmetrised $(-\Delta +\tau/2)u_+$ as above, for $u_-$ we followed the traditional approach of Talenti, and symmetrised $-\Delta u_-$. This hybrid
procedure allows us to preserve most of the boundary information regarding $u_-$. Of course, in return, it decreases more dramatically the energy, since the lower order
terms associated with $u_-$ have to be discarded. Overall, the operation gives rise to the asymetric two-ball problem
$$
-T(\Omega,\tau)\geq\min_{v,w}\int_{B_a}(\Delta v)^2+\tau\int_{B_a}|\nabla v|^2-2\int_{B_a}v+\int_{B_b}(\Delta w)^2,
$$
where $v\in H_0^1\cap H^2(B_a)$ and $w\in H_0^1\cap H^2(B_b)$ satisfy the \emph{incomplete coupling condition} $v,w\mygeq0$ and,
$\int_{\partial B_a}\partial_n v\geq\int_{\partial B_b}\partial_n w$. The balls $B_a$ and $B_b$ appearing in the energy are, as usual, defined by
$\left|B_a\right|=\left|\{u_+>0\}\right|$ and $\left|B_b\right|=\left|\{u_->0\}\right|$. Once again, a careful study of the minimising pairs $(v,w)$
allowed us to complete
the boundary coupling condition, and to prove that we have $\int_{\partial B_a}\partial_n v=\int_{\partial B_b}\partial_nw$ (actually the situation is slightly more complicated, see Proposition~\ref{prop:derivation pb 2 boule complet}).

Observe that, when $B_b=\Omega^*$, the asymmetric two-ball problem (with completed coupling condition) amounts to $0$, while when $B_a=\Omega^*$, its value is $-T(\Omega^*,\tau)<0$. In fact, after tedious computations, we were able to prove that the completed two-ball problem decreases with $a$ (see Proposition \ref{prop:pb 2 boules torsion}). As a result, we obtain the desired Saint--Venant type inequality.

Let us mention that the above philosophy of treating in a non-symmetric fashion the two-ball problem has become quite frequent recently~\cite{leylekian2,ashbaugh-bucur-laugesen-leylekian}, and seems to be an appropriate way to address this type of problems. Unfortunately, it would not be possible to apply this method in the case of the first eigenvalue. This is because the asymmetric quotient obtained after the hybrid factorisation method would then be
$$
\frac{\int_{B_a}(\Delta v)^2+\tau\int_{B_a}|\nabla v|^2+\int_{B_b}(\Delta w)^2}{\int_{B_a}v^2+\int_{B_b}w^2}.
$$
Now, even if one were able to complete the boundary coupling condition, the resulting two-ball problem would not provide a good enough lower bound, since for $B_b=\Omega^*$ its value would be $\Gamma(\Omega^*,0)<\Gamma(\Omega^*,\tau)$.

\subsection*{Ancillary results and structure of the paper}

Besides the Szeg\H{o}- and the Saint--Venant-type results, in this paper we also aimed at providing a rather broad spectral picture of
this class of mixed-order operators. In particular, after the preliminaries of Section \ref{sec:preliminaires}, we study the behaviour of the
spectrum with respect to the tension parameter $\tau$ in Section \ref{sec:behaviour wrt tension}. This leads us to several bounds on the first eigenvalue. The most
interesting is probably that of Corollary \ref{corollaire:triangles}, according to which triangles have a greater eigenvalue than balls with the same area, for any non-negative value of $\tau$. A similar conclusion is obtained in Corollary~\ref{corollaire:convexes} for elongated enough planar convex sets.

On the other hand, we analyse the sensitivity of the spectrum with respect to deformations of the domain in Section \ref{sec:behaviour wrt shape}. Since 
mixed-order operators are not homogeneous in general, interesting phenomena occur. The major part of the section is devoted to the theory of shape 
derivatives for simple eigenvalues of $S_{\tau}$. In particular, we compute the optimality condition associated with volume preserving deformations.

After this broad insight, we address our first main result, the Szeg\H{o}-type inequality, in Section~\ref{sec:szego}. We first derive the two-ball problem with incomplete
boundary condition (§~\ref{subsec:pb 1 boule}), after which the boundary condition is completed (§~\ref{subsec:pb 1 boule complet}), yielding the proof of
Theorem \ref{thmsz}.

Section~\ref{sec:torsion} is finally devoted to the Saint--Venant-type result of Theorem~\ref{thmsv}. To that end, we extract the two-ball problem, and complete the coupling condition in §~\ref{subsec:pb 2 boules complet}. Then we solve the completed two-ball problem in §~\ref{subsec:pb 2 boules}, and this finally leads to the proof
of Theorem~\ref{thmsv}, in §~\ref{subsec:saint-venant}. To conclude, we make explicit the bound obtained by investigating problem \eqref{torsion-eq} over the ball. There, the torsion function and the torsional rigidity are obtained in §~\ref{subsec:torsion}.

\subsection*{Notations}

Unless otherwise mentioned, we will use the following notations.
\renewcommand\labelitemi{$\cdot$}

\begin{itemize}
\item $\R$: set of real numbers.
\item $\N$: set of natural numbers, including $0$, that is $\N=\{0,1,\dots\}$. The set of positive natural numbers is denoted $\N^*$.
\item $d$: dimension of the ambient space ($d\in\N^*$).
\item $\Omega$: open subset of $\R^d$.
\item $B$: a Euclidean ball of $\R^d$.
\item $B_r$: the Euclidean ball of $\R^d$ of radius $r$, centred at the origin.
\item $\mathbb{S}^{d-1}$: boundary of $B_1$.
\item $|\cdot|$: $d$-dimensional Lebesgue measure or $(d-1)$-dimensional Hausdorff measure, depending on the situation.
\item $\Delta$: \enquote{analysts} Laplacian, defined by $\Delta=\dsum_{i=1}^d\partial_i^2$, where  $\partial_i$ is the directional derivative in the direction of the $i$-th coordinate.
\item $\Delta^2$: bilaplacian, defined by $\Delta^2 u =\Delta(\Delta u)$.
\item $\tau$: tension parameter ($\tau\in\R$).
\item $\partial_n$: outward normal derivative.
\item $L^p(\Omega)$: Lebesgue space of exponent $p$.
\item $C^{k,\alpha}(E)$: space of continuously differentiable functions in the interior of $E$, up to order $k$, whose derivatives shall be extended by continuity to the whole $E$, and with $\alpha$-Hölder continuous $k$-th derivatives.
\item $W^{m,p}(\Omega)$: Sobolev space of functions whose derivatives of order less or equal than $m$ are in $L^p(\Omega)$.
\item $H^m(\Omega)$: simplified notation for the Sobolev space $W^{m,2}(\Omega)$.
\item $H_0^m(\Omega)$: Sobolev space of functions in $H^m(\Omega)$ vanishing together with their derivatives up to order $m-1$ on the boundary. More formally, $H_0^m(\Omega)$ is the closure of $C_c^\infty(\Omega)$ with respect to $H^m(\Omega)$.
\item $S_\tau$: mixed order operator, defined by $\Delta^2-\tau\Delta$, and acting over the domain $H_0^2(\Omega)$.
\item $\Gamma(\Omega,\tau)$: first eigenvalue of $S_\tau$ over $\Omega$.
\item $u$: first eigenfunction of $S_\tau$.
\item $T(\Omega,\tau)$: torsional rigidity associated with $S_\tau$ over $\Omega$.
\item $w$: torsion function of $S_\tau$.
\item $\Gamma(\Omega)$: simplified notation for $\Gamma(\Omega,0)$.
\item $\Lambda(\Omega)$: first buckling eigenvalue over $\Omega$.
\item $\lambda(\Omega)$: first eigenvalue of the Dirichlet Laplacian over $\Omega$.
\item $^*$: Schwarz symmetrisation.
\item $\myleq,\mygeq$: relation of concentration between radially symmetric real-valued functions.
\item $\nu=d/2-1$.
\item $J_\mu,Y_\mu$: Bessel functions of order $\mu$ of first and second kind, respectively.
\item $I_\mu,K_\mu$: modified Bessel functions of order $\mu$ of first and second kind, respectively.
\item $(j_{\mu,i})_{i\in\N^*}$: positive zeros of $J_\mu$, counted in increasing order.
\item $j_\mu$: simplified notation for $j_{\mu,1}$.
\item $S_\kappa$: spherical harmonic of index $\kappa\in\N$ over $\mathbb{S}^{d-1}$.
\item $\gamma_\mu$: smallest positive zero of $J_{\mu+1}/J_\mu+I_{\mu+1}/I_\mu$.
\end{itemize}

\section{Preliminaries}\label{sec:preliminaires}

Let $\Omega$ be an open set of $\R^d$ and $\tau\in\R$. In this section, we do not make any sign assumption on~$\tau$, unless otherwise mentioned.

\subsection*{Spectrum of $S_{\tau}$}
We say that $\Gamma$ is an eigenvalue of $S_\tau$ and that $u$ is an associated eigenfunction if $u\in H_0^2(\Omega)$ satisfies the first equation of problem \eqref{eq:pb vp} in the sense of distributions. This means that, for all $v\in H_0^2(\Omega)$,
$$
\int_\Omega\Delta u\Delta v+\tau\int_\Omega \nabla u\nabla v=\Gamma\int_\Omega uv.
$$
When $\Omega$ is of finite measure, $S_\tau$ admits a sequence of eigenvalues with finite multiplicity, converging to infinity. This property is a consequence of the fact that the corresponding quadratic form is bounded from below:
$$
\int_\Omega(\Delta u)^2+\tau\int_\Omega|\nabla u|^2\geq c\int_\Omega u^2,
$$
for some $c\in\R$ . While this is immediate when $\tau\geq0$ by Poincaré's inequality, when $\tau<0$ it follows from applying the Cauchy-Schwarz inequality and minimising the resulting second-order polynomial in $\|\Delta u\|_{L^2(\Omega)}$:
$$
\int_\Omega(\Delta u)^2+\tau\int_\Omega|\nabla u|^2\geq\int_\Omega(\Delta u)^2+\tau\sqrt{\int_\Omega(\Delta u)^2\int_\Omega u^2}\geq-\frac{\tau^2}{4}\int_\Omega u^2.
$$
That the quadratic form is bounded from below allows to construct a bounded resolvent operator on $L^2(\Omega)$. Furthermore, the resolvent is compact, which follows from the fact that $\|\Delta u\|_{L^2(\Omega)}$ defines a norm whenever $u\in H_0^2(\Omega)$, and from the fact that $H_0^2(\Omega)$ is compactly embedded in $L^2(\Omega)$ whenever $\Omega$ has finite volume. As a consequence, $S_\tau$ has a discrete spectrum and the first eigenvalue $\Gamma(\Omega,\tau)$ (or $\Gamma(\tau)$ if there is no possible confusion) is well-defined. It can be characterised variationally by
\begin{equation}\label{eq:caracterisation variationnelle}
\Gamma(\Omega,\tau)=\min_{\substack{v\in H_0^2(\Omega)\\v\neq0}}\fr{\dint_\Omega(\Delta v)^2+\tau\dint_\Omega|\nabla v|^2}{\dint_\Omega v^2}.
\end{equation}
Observe that for $\tau\geq-\Lambda(\Omega)$, where $\Lambda(\Omega)$ is the first buckling eigenvalue (see below), the eigenvalue $\Gamma(\Omega,\tau)$ is non-negative. Note also that any function realizing the above quotient is an eigenfunction associated with $\Gamma(\Omega,\tau)$. Let us point out that, as usual with fourth order operator \cite{duffin-shaffer,coffman-duffin-shaffer,coffman,kozlov-kondratiev-mazya,coffman-duffin,wieners,brown-davies-jimack-mihajlovic}, there is no reason to assume $\Gamma(\Omega,\tau)$ to be simple and the corresponding eigenfunctions to be one-sign. Actually, as $\tau$ varies this property may fail even on domains as nice as balls \cite{coster-nicaise-troestler15}.

\subsection*{Buckling problem}

Although not directly concerned by Theorems~\ref{thmsz} and~\ref{thmsv}, the buckling problem enjoys close links with the operator $S_{\tau}$. The 
eigenvalues of the buckling problem (or buckling eigenvalues) over $\Omega$, are defined as the real numbers $\Lambda$ for which the problem
\begin{equation}\label{eq:pb vp buckling}
\left\{
\begin{array}{rcll}
\Delta^2 u & = & -\Lambda\Delta u & \text{in }\Omega,\\
u & = & 0 &\text{on }\partial\Omega,\\
\partial_n u & = & 0 &\text{on }\partial\Omega,
\end{array}
\right.
\end{equation}
admits a nontrivial solution $u:\Omega\to\R$, called an associated buckling eigenfunction. As previously, this has to be understood weakly, in the sense that $u\in H_0^2(\Omega)$ satisfies
$$
\int_\Omega\Delta u\Delta v=\Lambda\int_\Omega \nabla u\nabla v,\qquad\forall v\in H_0^2(\Omega).
$$
Again, whenever $\Omega$ has finite volume one can show that buckling eigenvalues form a sequence of real numbers converging to infinity. Note however that, unlike the eigenfunctions of $S_\tau$, which form a complete orthonormal system in $L^2(\Omega)$, the buckling eigenfunctions form a complete orthonormal system in $H^1(\Omega)$ (this has surprising consequences, such as the fact that in one dimension the even eigenfunctions do not change sign \cite[equation (10)]{grunau}). This is because the resolvent, mapping a given data $f\in H^1(\Omega)$ to the solution $u\in H_0^2(\Omega)$ of the equation $\Delta^2u=-\Delta f$, defines a compact endomorphism of $H^1(\Omega)$. In any case, the smallest buckling eigenvalue is well-defined. We denote it by $\Lambda(\Omega)$, and recall that it admits the following 
variational formulation:
\begin{equation}\label{eq:caracterisation variationnelle buckling}
\Lambda(\Omega)=\min_{\substack{v\in H_0^2(\Omega)\\v\neq0}}\fr{\dint_\Omega(\Delta v)^2}{\dint_\Omega |\nabla v|^2}.
\end{equation}
This shows, by Poincaré's inequality, that $\Lambda(\Omega)>0$. Together with~\eqref{eq:caracterisation variationnelle}, the previous formula also confirms that for any $\tau\geq-\Lambda(\Omega)$, the first eigenvalue $\Gamma(\Omega,\tau)$ of $S_\tau$ is non-negative. As before, note that $\Lambda(\Omega)$ may be multiple and that the first eigenfunctions may change sign \cite{kozlov-kondratiev-mazya,buoso-parini}. We refer to \cite{coster-nicaise-troestler15} for a detailed study of the buckling problem over the ball, where the buckling eigenvalues and eigenfunctions are computed explicitly. See \cite{buoso-lamberti13,buoso-freitas} for further information and more general developments on buckling type problems.

\subsection*{Torsion function and torsional rigidity}
The torsion function $w$ is defined as the function in $H_0^2(\Omega)$ satisfying the first equation of problem \eqref{torsion-eq} in the sense of distributions. This means that, for all $v\in H_0^2(\Omega)$,
$$
\int_\Omega\Delta w\Delta v+\tau\int_\Omega\nabla w\nabla v=\int_\Omega v. 
$$
Existence and uniqueness for the torsion function shall be proved, by the Fredholm alternative, whenever $-\tau$ is not an eigenvalue of the buckling problem. When $-\tau$ is a buckling eigenvalue, the situation is more complicated, and either existence or uniqueness may fail (see in particular §~\ref{subsec:torsion}, where the torsion function is computed over balls). Note that, as usual, $w$ is in general sign-changing \cite{grunau-sweers}. Then, we define the torsional rigidity, also called sometimes compliance or mean deflection, by
\begin{equation*}\label{eq:rigidité torsionnelle}
 T(\Omega,\tau) = \dint_{\Omega} w.
\end{equation*}
Whenever $\tau>-\Lambda(\Omega)$, it is also interesting to observe that the torsion corresponds to the minimiser of problem
$$
\min_{u\in H_0^2(\Omega)}\int_\Omega(\Delta u)^2+\tau\int_\Omega|\nabla u|^2-2\int_\Omega u.
$$
The important point is that, for $\tau$ in this range, the above energy is coercive, hence reaches its minimum. By computing the value of this energy at $w$, we realize that it is related to the torsional rigidity in the following way:
\begin{equation}\label{eq:caracterisation variationnelle torsion}
-T(\Omega,\tau)=\min_{u\in H_0^2(\Omega)}\int_\Omega(\Delta u)^2+\tau\int_\Omega|\nabla u|^2-2\int_\Omega u.
\end{equation}
Let us also point out, even if this will not be used in this paper, that the torsion function and the torsional rigidity are related to the following eigenproblem:
$$
\left\{
\begin{array}{rcll}
\Delta^2 u -\tau\Delta u & = & \mu\int_\Omega u & \text{in }\Omega,\\
u & = & 0 &\text{on }\partial\Omega,\\
\partial_n u & = & 0 &\text{on }\partial\Omega,
\end{array}
\right.
$$
Indeed, whenever the torsion function $w$ exists and is unique, the above problem admits only one eigenvalue $\mu=T(\Omega,\tau)^{-1}$ and any associated eigenfunction is a multiple of $w$. Interestingly, this provides an alternative variational formulation for the torsional rigidity:
$$
T(\Omega,\tau)^{-1}=\min_{\substack{v\in H_0^2(\Omega)\\v\neq0}}\fr{\dint_\Omega(\Delta v)^2+\tau\dint_\Omega|\nabla v|^2}{\left(\dint_\Omega v\right)^2}.
$$
This point of view was adopted in \cite[Section~5.5]{leylekian-these}, but in the present paper we rather follow the approach of \cite{ashbaugh-bucur-laugesen-leylekian}, based on formula \eqref{eq:caracterisation variationnelle torsion}. Yet the previous connection justifies to view the torsional rigidity as a spectral quantity.

\subsection*{Symmetrisation techniques}
The variational characterizations \eqref{eq:caracterisation variationnelle} and \eqref{eq:caracterisation variationnelle torsion} allow to employ techniques steming from the calculus of variations. One of them is the process of level-set rearrangement, or more concisely symmetrisation. Let us recall some basic facts on this tools. The symmetrisation of a measurable subset $\omega$ of $\R^d$ is the open ball $\omega^*$ centered at the origin, with volume $|\omega|$. Then, for a function $f:\Omega\to\R^d$, its symmetrisation is a radially symmetric and non-increasing function $f^*:\Omega^*\to\R$ whose upper level sets satisfy
$$
\{f^*>t\}=\{f>t\}^*\quad a.e.,
$$
for all $t\in\R$. We emphasise that, in the wake of \cite{kesavan,baernstein}, our definition is compatible with sign-changing functions, in the sense that $\int_\Omega f=\int_{\Omega^*} f^*$. On the contrary, the traditional definition of Schwarz symmetrisation (see e.g. \cite{talenti76,alvino-lions-trombetti90}) would correspond to $|f|^*$, but this would preserve integration only for positive functions. Since we are working with sign-changing functions in this document, the definition proposed above is more suitable, as was first noted in \cite{talenti81}.

In the present paper, the central tool with respect to level set rearangement is a version of the famous comparison principle of Talenti \cite{talenti76} for elliptic operators with a zero order term. In this case, the comparison result is weaker than the classic one (for which see e.g. \cite[Theorem~3.1.1]{kesavan}) and it involves the order relation $\myleq$ defined for two radially symmetric functions $f$ and $g$ by:
\begin{equation}\label{eq:concentration}
f\myleq g\quad\Leftrightarrow\quad \int_{B_r}f\leq\int_{B_r}g,\quad\forall r\geq0.
\end{equation}
When the above relation holds, we say that $f$ is less concentrated than $g$, and $g$ is more concentrated than $f$. This is closely related to the well-known relation of domination $\preceq$
introduced in \cite{alvino-lions-trombetti89}, and defined by $f\preceq g \Leftrightarrow f^*\myleq g^*$. In particular, the relation of concentration enjoys similar properties  to those of that
relation. For instance, if $f^*\myleq g$ then one can compare the $L^p$ norms of the positive parts of $f$ and $g$. More generally, for any convex non-decreasing function $\Phi:\R\to\R$,
\begin{equation*}\label{eq:comparaison Phi}
\int\Phi(f)\leq\int\Phi(g).
\end{equation*}
Following~\cite[Proposition 2.1]{alvino-lions-trombetti89}, this may be obtained by noting that $\Phi(g)-\Phi(f^*)\geq\Phi'(f^*)(g-f^*)=\Phi'(f)^*\partial_r\left[\int_{B_{|x|}}(g-f^*)\right]$ by convexity of $\Phi$, and then integrating by parts over a ball $B$:
$$
\int_{B}(\Phi(g)-\Phi(f^*))\geq \int_{\partial B}\Phi'(f)^*\int_{B}(g-f^*)-\int_{B}\left(\partial_r[\Phi'(f)^*](x)\int_{B_{|x|}}(g-f^*)\right)dx\geq0,
$$
since $\Phi'(f)^*$ is non-negative and radially non-increasing. A comparison result for elliptic operators with lower order terms was proved for the first time in \cite{alvino-lions-trombetti90}. But since we need a statement compatible with sign-changing functions, we rather refer to~\cite[Theorem 10.10]{baernstein}:

\begin{theoreme}[\cite{baernstein}, Theorem 10.10]\label{thm:talenti}
Let $\Omega$ be an open subset of $\R^d$ of finite volume, let $f\in L^2(\Omega)$, and let a real number $\sigma>-\lambda(\Omega^*)$, where $\lambda(\Omega^*)$ is the first eigenvalue of the Dirichlet Laplacian over $\Omega^*$. Consider the solutions $u\in H_0^1(\Omega)$ and $v\in H_0^1(\Omega^*)$ of the problems
$$
\left\{
\begin{array}{rcll}
-\Delta u+\sigma u & = & f &\text{in }\Omega,\\
u & = & 0 &\text{on }\partial\Omega,
\end{array}
\right.
\qquad
\left\{
\begin{array}{rcll}
-\Delta v+\sigma v & = & f^* &\text{in }\Omega^*,\\
v & = & 0 &\text{on }\partial\Omega^*,
\end{array}
\right.
$$
Assume furthermore that $u$ is non-negative. Then, $u^*\myleq v$, and hence for all $p\in[1,\infty[$,
$$
\int_\Omega |u|^p\leq\int_{\Omega^*}|v|^p.
$$
\end{theoreme}

\section{Behaviour of the spectrum with respect to $\tau$}\label{sec:behaviour wrt tension}

In this part, we fix a domain $\Omega$ of finite volume, and study the sensitivity of $\Gamma(\Omega,\tau)$ with respect to the tension parameter $\tau\in\R$. First, as a consequence
of~\eqref{eq:caracterisation variationnelle}, several properties can be easily deduced. Indeed, being the infimum of linear non-decreasing functions, $\Gamma(\Omega,\tau)$ is concave and
non-decreasing with respect to $\tau$. We will see in a few moment that it is actually increasing. Denote by $u_\tau$ an $L^2$-normalised eigenfunction associated with $\Gamma(\tau)$. According to
Kato's theory of analytic perturbation of operators \cite[Chapter VII, \S~3]{kato}, $u_\tau$ can be extended into a family of eigenfunctions associated with $\Gamma(\tau)$, both depending analytically on $\tau$ except at points where the multiplicity of $\Gamma(\tau)$ changes, but those points remain isolated. Furthermore, we have
$$
\left\{
\begin{array}{rcll}
\Delta^2 u_\tau'-\tau\Delta u_\tau'-\Delta u_\tau & = & \Gamma(\tau) u_\tau'+\Gamma'(\tau) u_\tau & \text{in }\Omega,\\
u_\tau' & = & 0 &\text{on }\partial\Omega,\\
\partial_n u_\tau' & = & 0 &\text{on }\partial\Omega,
\end{array}
\right.
$$
together with the condition $\int_\Omega u_\tau'u_\tau=0$, where $'$ denotes the derivative with respect to $\tau$. After multiplying by $u_\tau$ and integrating by parts, we conclude, by \eqref{eq:pb vp} that
\begin{equation}\label{eq:derivee en tau}
\Gamma'(\tau)=\int_\Omega|\nabla u_\tau|^2
\end{equation}
Immediately, we observe that $\Gamma'(\tau)$ is not smaller than the first eigenvalue $\lambda(\Omega)$ of the Dirichlet Laplacian. In particular, $\Gamma(\cdot)$ is increasing with respect to $\tau$. By concavity of $\Gamma(\cdot)$, we also see that for non-negative values of $\tau$, $\Gamma'(\tau)$ is not greater than the gradient of any first $L^2$-normalised eigenfunction of the Dirichlet bilaplacian. For readability, we will denote by $\Gamma(\Omega):=\Gamma(\Omega,0)$ the associated first eigenvalue.

Lastly, standard arguments from perturbation theory applied to the operator $\tau^{-1}\Delta^2-\Delta$ (see e.g. \cite[section 4]{buoso-kennedy} or \cite[section 2.3.1]{henrot}) show that
the eigenvalues of the operator $\tau^{-1} S_{\tau}$, converge towards those of the Dirichlet Laplacian as $\tau\to+\infty$ (moreover, the eigenfunctions 
converge weakly in $H_0^1(\Omega)$ toward the eigenfunctions of the Laplacian). Note that when $\Omega$ is smooth, bounded and connected, it is even 
possible to refine the asymptotic behaviour \cite[p.~392]{frank}:
$$
\Gamma(\tau)\underset{\tau\to+\infty}{=}\lambda(\Omega)\tau+\sqrt{\tau}\int_{\partial\Omega}(\partial_n u_\infty)^2+O(1),
$$
where $u_\infty$ is the first normalised eigenfunction of the Dirichlet Laplacian. The previous discussion can be summarised in the following lemma (see also Figure~\ref{fig:comportement en tau}). We emphasise, however, that this lemma is not new as it can be found in \cite[Remark 4]{kawohl-levine-velte}.

\begin{lemme}\label{lemme:comportement en tau}
Let $\Omega$ be an open set of finite measure and let $u_0$ be a first $L^2$ normalised eigenfunction of $S_0$ on $\Omega$.
\begin{enumerate}
\item $\Gamma(\Omega,\cdot)$ is concave, inscreasing, and piecewise analytic over $\R$.
\item For all $\tau\geq0$,
$
\lambda(\Omega)\leq\partial_\tau\Gamma(\Omega,\tau)\leq\dint_\Omega|\nabla u_0|^2
$.
\item $\Gamma(\Omega,\tau)\underset{\tau\to+\infty}{\sim}\lambda(\Omega)\tau$.
\end{enumerate}
\end{lemme}

\begin{remarque}
We insist on the fact that it is likely for $\Gamma(\Omega,\cdot)$ not to be smooth over the whole real line, as one cannot exclude the possibility of a multiple first eigenvalue for some values of $\tau$.
\end{remarque}

\begin{proof}
The first two points follow from the above discussion (see in particular~\eqref{eq:derivee en tau}), hence we focus on the last point. To address it, we let $f\in L^2(\Omega)$ and consider $v_\tau\in H_0^2(\Omega)$ the solution of $\Delta^2v_\tau/\tau-\Delta v_\tau=f$ in $\Omega$. Multiplying the equation by $v_\tau$ and integrating by parts shows that $v_\tau$ remains bounded in $H_0^1(\Omega)$ uniformly with respect to $\tau>0$. Therefore, as $\tau\to+\infty$, $v_\tau$ converges weakly in $H_0^1(\Omega)$ and strongly in $L^2(\Omega)$ to some $v_\infty$, which must satisfy $-\Delta v_\infty=f$. In other words, the resolvent of $S_\tau/\tau$, as an operator from $L^2(\Omega)$ to $L^2(\Omega)$, simply converges to the resolvent of the Dirichlet Laplacian. By \cite[Theorem~2.3.2]{henrot}, we conclude that the convergence also holds in operator norm, and hence that the eigenvalues of $S_\tau/\tau$ converge to the eigenvalues of the Dirichlet Laplacian, which concludes.
\end{proof}

\begin{figure}
\includegraphics{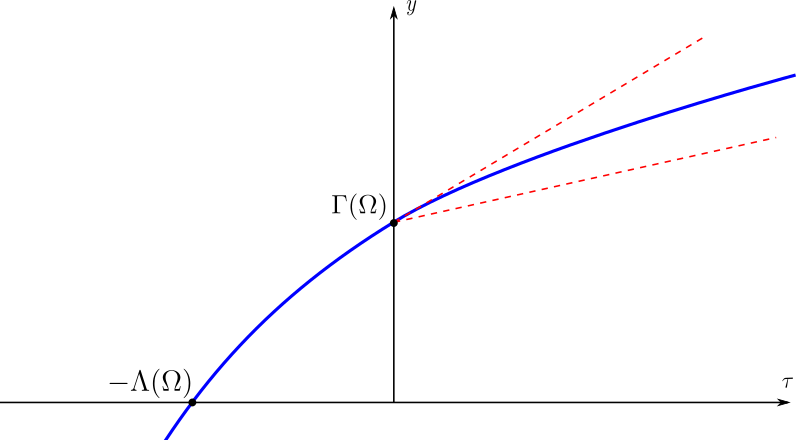}
\caption{Behaviour of the first eigenvalue $\Gamma(\Omega,\tau)$ (in blue) with respect to the tension parameter $\tau$. In dotted red, above the blue curve, a portion of the line of equation $y=\Gamma(\Omega)+\tau\int_\Omega|\nabla u_0|^2$. Also in dotted red, below the blue curve, a portion of the line $y=\Gamma(\Omega)+\tau\lambda(\Omega)$. The blue curve and the two dotted red lines cut the ordinates at $y=\Gamma(\Omega)$. The blue curve cuts the abcissa line at $-\tau=\Lambda(\Omega)$, the first buckling eigenvalue.}
\label{fig:comportement en tau}
\end{figure}

Thanks to the previous lemma, it is possible to obtain elementary bounds on the first eigenvalue (see also \cite{payne} in this respect). Indeed, one shall compare the blue curve of Figure~\ref{fig:comportement en tau} in the case of an arbitrary open set $\Omega$, and in the case of the ball $B$, only by comparing the associated dotted red lines. If the slope of the lower dotted red line for $\Omega$ is greater than the slope of the upper dotted red line for $B$, it means that the blue curve associated with $\Omega$ is eventually above the blue curve associated with $B$. This is basically what is stated in the next corollary.

\begin{corollaire}\label{corollaire:comportement vp en tau}
Let $\Omega$ be an open set of finite measure and $B$ an open ball with the same volume. Assume that $\lambda(\Omega)>\int_B|\nabla u_{B,0}|^2$, where $u_{B,0}$ is a first normalised eigenfunction of the bilaplacian over $B$. Then, we have the inequality
$$
\Gamma(\Omega,\tau)\geq\Gamma(B,\tau)
$$
as long as $\tau\geq\max(\tau_\Omega,0)$, where
$$
\tau_\Omega:= \frac{\Gamma(B)-\Gamma(\Omega)}{\lambda(\Omega)-\int_B|\nabla u_{B,0}|^2}.
$$
\end{corollaire}

\begin{proof}
If $\lambda(\Omega)>\int_B|\nabla u_{B,0}|^2$, for any $\tau\geq\tau_\Omega$ we have
$$
\Gamma(\Omega)+\tau\lambda(\Omega)\geq\Gamma(B)+\tau\frac{\int_B|\nabla u_{B,0}|^2}{\int_B u_{B,0}^2},
$$
and we conclude thanks to the second point of Lemma \ref{lemme:comportement en tau}.
\end{proof}

Using the previous bound, one may prove new results for some special shapes. To that end, it is first needed to determine explicitly the norm of the gradient of $u_{B,0}$, which is performed in Appendix (see Lemma \ref{lemme:norme L2 grad u0}). As a consequence, we are able to show that triangles cannot minimise the first eigenvalue of $S_{\tau}$, for any $\tau\geq0$.

\begin{corollaire}[Triangles]\label{corollaire:triangles}
If $T$ is a triangle and $B$ a disk with the same area, then for any $\tau\geq0$
$$
\Gamma(T,\tau)\geq\Gamma(B,\tau).
$$
\end{corollaire}

\begin{proof}
By Corollary \ref{corollaire:comportement vp en tau} and  Lemma \ref{lemme:norme L2 grad u0}, the inequality holds for any $\tau\geq\max(\tau_T,0)$ as long as
\begin{equation}\label{eq:critère pentes d=2}
R^2\lambda(T)>\gamma_0^2\frac{J_{1}(\gamma_0)^2}{J_0(\gamma_0)^2}\simeq 6.92801,
\end{equation}
where $R$ is the radius of $B$, and $\gamma_0\simeq 3.19622$ is defined in Appendix~\ref{annexe:fonction propre boule} as the first positive zero of a cross-product of Bessel functions. Note that we used the identity $J_1+J_{-1}=0$. Due to the inequality $\Gamma(T)\geq\Gamma(B)$ \cite{nadirashvili}, we have $\tau_T\leq0$, hence it remains only to check inequality \eqref{eq:critère pentes d=2}. By the polygonal Faber-Krahn inequality (for triangles) it is enough to consider the case of the equilateral triangle $T^*$. By homogeneity, we assume that $R=1$. In this case, the inner radius of $T^*$ is $r=3^{-3/4}$, and we have (see \cite[equation (6.18)]{mccartin})
$$
\lambda(T^*)=\frac{4\pi^2}{27r^2}=\frac{4}{3\sqrt{3}}\pi^2\simeq7.5976
$$
This proves inequality \eqref{eq:critère pentes d=2}, and hence the Faber-Krahn type inequality.
\end{proof}

As another application of Corollary \ref{corollaire:comportement vp en tau}, we mention the following result, which addresses the case of planar convex sets. It tells that when the set is elongated enough (which is expressed in terms of its perimeter), the corresponding first eigenvalue is greater or equal than that of a ball of same area, regardless of $\tau\geq0$.

\begin{corollaire}[Planar convex domains]\label{corollaire:convexes}
Let $\Omega$ be a planar convex bounded open set and $B$ an open disk with the same area. Then, for any $\tau\geq0$ we have the inequality
$$
\Gamma(\Omega,\tau)\geq\Gamma(B,\tau),
$$
as long as
$$
\frac{\pi|\partial\Omega|^2}{16|\Omega|}>\gamma_0^2\frac{J_{1}(\gamma_0)^2}{J_0(\gamma_0)^2}\simeq 6.92801,
$$
where $\gamma_0$ is defined in Appendix~\ref{annexe:fonction propre boule}.
\end{corollaire}

\begin{proof}
Makai's inequality for convex sets~\cite{makai} reads
$$
\lambda(\Omega)\geq \frac{\pi^2|\partial\Omega|^2}{16|\Omega|^2}.
$$
Then, by Corollary \ref{corollaire:comportement vp en tau} and Lemma \ref{lemme:norme L2 grad u0}, the Faber-Krahn type inequality holds for all $\tau\geq\max(\tau_\Omega,0)$ as long as
$$
\frac{\pi^2|\partial\Omega|^2}{16|\Omega|^2}>\frac{\gamma_0^2J_{1}(\gamma_0)^2}{R^2J_0(\gamma_0)^2},
$$
where $R$ is the radius of $B$. But $\tau_\Omega\leq0$ since $\Gamma(\Omega)\geq\Gamma(B)$ \cite{nadirashvili}. This concludes, after elementary manipulations of the previous inequality.
\end{proof}

\section{Behaviour of the spectrum with respect to $\Omega$}\label{sec:behaviour wrt shape}

Besides the dependence of $\Gamma(\Omega,\tau)$ with respect to the parameter $\tau$ studied in the previous section, it is interesting to analyse its behaviour with respect to the domain $\Omega$. First, let us mention that the functional $\Gamma(\Omega,\tau)$ is not homogeneous with respect to $\Omega$. Indeed, for any $\alpha>0$,
\begin{equation}
\Gamma\left(\alpha\Omega,\frac{\tau}{\alpha^2}\right)=\alpha^{-4}\Gamma(\Omega,\tau)
\end{equation}
This might appear problematic since homogeneity with respect to dilations is an important property of the operators classically addressed. Nevertheless, it is still possible to define the scale-invariant companion functional
\begin{equation}\label{eq:gamma}
\gamma(\Omega,\tau):=|\Omega|^{\frac{4}{d}}\Gamma\left(\Omega,|\Omega|^{-\frac{2}{d}}\tau\right).
\end{equation}
Observe that a set is optimal for problem \eqref{eq:pb min}, with the volume constraint $c\in\R$, if and only if it is optimal for the constraint-free functional $\gamma(\cdot,c^{\frac{2}{d}}\tau)$. Among others, this allows to compute in a straightforward way the optimality condition corresponding to problem~\eqref{eq:pb min}. To this end, for a given domain $\Omega$ and a given parameter $\tau$, we define
$$
\tilde{\Gamma}[V]:=\Gamma((id+V)\Omega,\tau),
$$
where $V$ is some vector field over $\R^d$. The first step is to make sure that $\tilde{\Gamma}$ can be differentiated at $V=0$. This follows from strandard argument, but for completeness we provide a proof in Appendix~\ref{annexe:derivee de forme}. Indeed, we show in Lemma \ref{lemme:derivee de forme} that, as long as $\Omega$ is $C^4$ regular and $\Gamma(\Omega,\tau)$ is simple, $\tilde{\Gamma}[V]$ is differentiable with respect to $V\in W^{2,\infty}(\R^d,\R^d)$, together with the associated $L^2$-normalised eigenfunction $u[V]$. Furthermore, if $u'$ denotes the derivative of $u$ in the direction of $V$ at $0$, we have $u'\in H^2(\Omega)$ and
$$
u'=0\text{ on }\partial\Omega,\qquad \partial_nu'=-\partial_n^2u V\cdot\vec{n}\text{ on }\partial\Omega.
$$
Consequently, we may differentiate the eigenproblem \eqref{eq:pb vp} in the direction of $V$. Writing $\Gamma'$ for $\partial_V\tilde{\Gamma}[0]$, the differentiated equation reads
$$
\Delta^2u'-\tau\Delta u'=\Gamma u'+\Gamma' u.
$$
Multiplying by $u$ and integrating by parts, which is possible since $u\in H^4(\Omega)$ by elliptic regularity, we see that
$$
\int_{\partial\Omega}\partial_n u'\Delta u=\Gamma\int_\Omega u'u+\Gamma'.
$$
But recall that, as a consequence of the normalisation of $u$, $\int_\Omega u'u=0$. Then, using the value of $\partial_nu'$, we conclude that
\begin{equation}\label{eq:derivee de forme}
\partial_V\tilde{\Gamma}[0]=-\int_{\partial\Omega}(\Delta u)^2V\cdot\vec{n}.
\end{equation}
Therefore, we have obtained the derivative of $\Gamma(\Omega,V)$ \enquote{with respect to $\Omega$}. Note that this formula, together with more general results, is available in \cite{buoso}, where the case of multiple eigenvalues is also addressed. In order to differentiate $\gamma$ we also need, in view of \eqref{eq:gamma}, to compute the derivative of $\Gamma(\Omega,\tau)$ with respect to $\tau$. This was done in Section~\ref{sec:behaviour wrt tension}, where we saw that when $\Gamma(\Omega,\tau)$ is simple, $\Gamma(\Omega,\cdot)$ is smooth near $\tau$, and the actual derivative is given by \eqref{eq:derivee en tau}. Lastly, we need the derivative of the volume functional $\mathcal{V}[V]:=|(id+V)\Omega|$ with respect to $V$. For $V\in W^{1,\infty}(\R^d,\R^d)$, this is given by a classic formula (see \cite[Théorème 5.2.2]{henrot-pierre}):
\begin{equation}\label{eq:derivee volume}
\partial_V\mathcal{V}[0]=\int_{\partial\Omega}V\cdot\vec{n}.
\end{equation}

The combination of \eqref{eq:derivee de forme}, \eqref{eq:derivee en tau}, and \eqref{eq:derivee volume}, allows to compute the derivative of the functional $\tilde{\gamma}[V]:=\gamma((id+V)\Omega,\tau)$ at $0$ in the direction of $V\in W^{2,\infty}(\R^d,\R^d)$. Indeed, by the chain rule, putting $\sigma=\tau|\Omega|^{-\frac{2}{d}}$, we have
\begin{equation}\label{eq:derivee de forme homogene}
\partial_V\tilde{\gamma}[0]=|\Omega|^{\frac{4}{d}}\left(\frac{4\Gamma(\Omega,\sigma)}{d|\Omega|}\int_{\partial\Omega}V\cdot\vec{n}-\int_{\partial\Omega}(\Delta u_\sigma)^2V\cdot\vec{n}-\frac{2\sigma}{d|\Omega|}\int_{\partial\Omega}V\cdot\vec{n}\int_\Omega|\nabla u_\sigma|^2\right),
\end{equation}
where $u_\sigma$ is a normalised eigenfunction associated with $\Gamma(\Omega,\sigma)$, which we assume to be simple. As a result of this discussion, one may extract the optimality condition related to problem~\eqref{eq:pb min}.

\begin{theoreme}\label{thm:condition optimalite}
Let $\tau\in\R$ and let $\Omega$ be a $C^4$ bounded open set solving problem \eqref{eq:pb min}. Assume that the first eigenvalue $\Gamma(\Omega,\tau)$ is simple, and let $u$ be an associated $L^2$-normalised eigenfunction. Then, $\Delta u$ is constant equal to $\pm\alpha$ over each connected component of $\partial\Omega$, where
$$
\alpha:=\left(\frac{4\Gamma(\Omega,\tau)}{d|\Omega|}-\frac{2\tau}{d|\Omega|}\int_\Omega|\nabla u|^2\right)^{1/2}.
$$
\end{theoreme}

\begin{remarque}
When $\tau=0$, we get back the optimality condition for the clamped plate \cite[Theorem~3]{leylekian1}. When $\tau$ goes to $+\infty$, by virtue of Lemma \ref{lemme:comportement en tau}, $\alpha$ is equivalent (at least formally) to $\beta\tau^{1/2}$, where $\beta=\sqrt{\frac{2\lambda(\Omega)}{d|\Omega|}}$ is the constant in the optimality condition associated with the minimisation of the first eigenvalue of the Laplacian.
\end{remarque}

\begin{remarque}
Theorem~\ref{thm:condition optimalite} should be compared with the results in \cite[section~4]{buoso}, where a similar optimality condition is obtained, and where balls are shown to satisfy this optimality condition.
\end{remarque}

\begin{proof}
Since $\Omega$ minimises $\Gamma(\cdot,\tau)$ under the volume constraint $|\Omega|=c$, it also minimises the companion functional $\gamma(\cdot,c^{\frac{2}{d}}\tau)$, but this time without any constraint. As a result, $V\in W^{2,\infty}(\R^d,\R^d)\mapsto \tilde{\gamma}[V]=\gamma((id+V)\Omega,c^{\frac{2}{d}}\tau)$ reaches its minimum value at $V=0$. In particular, its differential must be zero. This means that for any $V\in W^{2,\infty}(\R^d,\R^d)$, the expression \eqref{eq:derivee de forme homogene} vanishes (with $\sigma=c^{\frac{2}{d}}\tau|\Omega|^{-\frac{2}{d}}=\tau$). As a result, the Laplacian of $u$ on the boundary is prescribed by
$$
(\Delta u)^2=\frac{4\Gamma(\Omega,\tau)}{d|\Omega|}-\frac{2\tau}{d|\Omega|}\int_\Omega|\nabla u|^2=\alpha^2\text{ on }\partial\Omega.
$$
But since $\Omega\in C^4$, elliptic regularity, Sobolev embeddings, and bootstrap arguments (see e.g. \cite[Lemma 1.5]{leylekian-these}) show that $u\in C^{3,\alpha}(\overline{\Omega})$ for all $0<\alpha<1$, and hence $\Delta u$ is continuous on the boundary. Therefore, it must be constant on each connected component of $\partial\Omega$, and this constant is $\pm\alpha$.
\end{proof}

Optimality conditions are traditionally used to obtain information on 
optimal shapes. This is well-known for second order operators, and some 
results have been obtained also for fourth order operators~\cite{mohr,bennett,payne-schaefer,dalmasso,willms-weinberger,stollenwerk-wagner,leylekian1}. However, we were not 
able to deduce interesting information from the optimality conditions 
here. For instance, observe that, whereas $\alpha/\Gamma(\Omega)$ does 
not depend on $\Omega$ when $\tau=0$, this is not the case otherwise. 
This prevents us from following the approach in \cite{leylekian1}.

\section{Szeg\H{o} type result}\label{sec:szego}

Let us now address the proof of our first main result, which is the Szeg\H{o}-type inequality of Theorem ~\ref{thmsz}. As already explained, the first step is to derive a one-ball problem with incomplete boundary conditions. This is the purpose of the next paragraph.

\subsection{Derivation of a one-ball problem}\label{subsec:pb 1 boule}

\begin{proposition}\label{prop:derivation pb 1 boule}
Let $\tau\geq0$ and let $\Omega$ be a bounded open set over which the first eigenfunction of $S_\tau$ is one-sign. Then,

$$
\Gamma(\Omega,\tau)\geq\min_{\substack{v\in H^2\cap H_0^1(\Omega^*)\setminus\{0\},\\v\text{ radial,}\\v\mygeq 0\text{, }\partial_nv\geq0}}\frac{\int_{\Omega^*}(\Delta v)^2+\tau\int_{\Omega^*}|\nabla v|^2}{\int_{\Omega^*}v^2}.
$$
\end{proposition}

\begin{proof}
We shall assume without loss of generality that the first eigenfunction $u$ is non-negative. Recall that since $\Delta^2-\tau\Delta=(\Delta-\tau/2)^2-\tau^2/4$, the eigenfunction $u$ satisfies
\begin{equation}\label{eq:factorisation}
(\Delta-\tau/2)^2u=(\Gamma+\tau^2/4)u,
\end{equation}
which justifies to introduce $\gamma:=\Gamma+\tau^2/4$. Here $\Gamma$ stands for $\Gamma(\Omega,\tau)$. Since $\Gamma$ is the smallest eigevalue of $\Delta^2-\tau\Delta$, we deduce that $\gamma$ is the smallest eigenvalue of $(\Delta-\tau/2)^2$, with $u$ as an associated eigenfunction. In other words,
$$
\gamma=\min_{u\in H_0^2(\Omega)\setminus\{0\}}\frac{\int_\Omega[(\Delta-\tau/2)u]^2}{\int_\Omega u^2},
$$
and $u$ is a minimiser of this quotient. To build a competitor for $u$ in $\Omega^*$, we now symmetrise \eqref{eq:factorisation} in the following way. We put $f:=(-\Delta +\tau/2) u$ and let $v$ solve
\begin{equation*}
\left\{
\begin{array}{rcll}
(-\Delta +\tau/2) v & = & f^* & \text{in }\Omega,\\
v & = & 0 &\text{on }\partial\Omega..
\end{array}
\right.
\end{equation*}
By elliptic regularity, $v$ belongs to $H^2\cap H_0^1(\Omega^*)$. Talenti's comparison principle with zero order term (Theorem~\ref{thm:talenti}), which can be applied since $u\geq0$, then indicates that $v\mygeq u^*$. As a result, not only
$$
\int_\Omega u^2\leq\int_{\Omega^*} v^2,
$$
but also $\partial_n v\geq0$, since by the divergence theorem
$$
\int_{\partial\Omega^*}\partial_n v=-\int_{\Omega^*}f^*+\frac{\tau}{2}\int_{\Omega^*}v=-\int_{\Omega}f+\frac{\tau}{2}\int_{\Omega^*}v=\int_{\partial\Omega}\partial_n u+\frac{\tau}{2}\(\int_{\Omega^*}v-\int_{\Omega}u\right)\geq0.
$$
Consequently, we obtain the lower bound
$$
\gamma\geq \frac{\int_{\Omega^*}[(\Delta-\tau/2)v]^2}{\int_{\Omega^*} v^2}=\frac{\int_{\Omega^*}(\Delta v)^2+\tau\int_{\Omega^*}|\nabla v|^2}{\int_{\Omega^*} v^2}+\tau^2/4.
$$
Since $v\mygeq u^*\geq0$ and $\partial_n v\geq0$, this is the desired result. Note that the minimum is achieved since $\{v\in H_0^1\cap H^2(\Omega^*)\text{ radial}:v\mygeq 0\text{, }\partial_n v\geq0\}$ is weakly closed in $H^2(\Omega^*)$. Indeed, $v\in H^2(A)\mapsto \partial_n v\in H^1(\partial A)$ is a compact operator whenever $A$ is, just like $\Omega^*$, a bounded regular open set.
\end{proof}

\begin{remarque}\label{rmq:echec FK avec fct propre signee}
Let us briefly discuss what would happen if the eigenfunction $u$ were changing sign. In that situation, one would need to apply the previous symmetrisation procedure separately to the restrictions $u|_{\Omega_+}$ and $(-u)|_{\Omega_-}$ of the eigenfunction on the positive and the negative nodal domains $\Omega_\pm=\{\pm u>0\}$. As a result, we would end up with two competitors $v\in H_0^1\cap H^2(\Omega_+^*)$ and $w\in H_0^1\cap H^2(\Omega_-^*)$, defined such that
$$
(-\Delta +\tau/2)v=[(-\Delta+\tau/2)u|_{\Omega_+}]^*,\qquad (-\Delta +\tau/2)w=[(-\Delta+\tau/2)(-u)|_{\Omega_-}]^*.
$$
Then, Theorem \ref{thm:talenti} would still guarantee $v\mygeq u|_{\Omega_+}^*\geq0$ and $w\mygeq(-u)|_{\Omega_-}^*\geq 0$. However the previous incomplete boundary condition would now become:
$$
\int_{\Omega_+^*}(-\Delta +\tau/2)v-\int_{\Omega_-^*}(-\Delta +\tau/2)w=\int_{\Omega}(-\Delta +\tau/2)u,
$$
or in other words
$$
\partial_n v|\partial\Omega_+^*|-\partial_n w|\partial\Omega_-^*|=\frac{\tau}{2}\left[\int_{\Omega_+^*}v-\int_{\Omega_-^*}w -\int_\Omega u\right].
$$
We now see that the sign of the term inside the brackets above cannot 
be determined from the information available, since the relations $v\mygeq u|_{\Omega_+}^*$ and $w\mygeq (-u)|_{\Omega_-}^*$ do not combine in an appropriate way. Therefore, unlike with traditional two-ball problems (see \cite[equation (30)]{ashbaugh-benguria}), we were unable to deduce (even an incomplete version of) the coupling condition between $\partial_nv$ and $\partial_nw$. Overall, this prevented us to address the case of a signed eigenfunction $u$.
\end{remarque}

\subsection{Completion of the one-ball problem}\label{subsec:pb 1 boule complet}

After the derivation of the incomplete one-ball problem, the purpose is now to complete the boundary condition in order to retrieve the variational characterization in \eqref{eq:caracterisation variationnelle}. This is performed below.

\begin{proposition}\label{prop:pb 1 boule}
Let $B$ be an open ball. Any solution of
\begin{equation}\label{eq:pb 1 boule}
\min_{\substack{v\in H^2\cap H_0^1(B)\setminus\{0\},\\v\text{ radial,}\\v\mygeq 0\text{, }\partial_nv\geq0}}\frac{\int_{B}(\Delta v)^2+\tau\int_{B}|\nabla v|^2}{\int_{B}v^2}
\end{equation}
belongs to $H_0^2(B)$.
\end{proposition}

\noindent
Thanks to this proposition, the proof of Theorem \ref{thmsz} follows in a straightforward way.

\begin{proof}[Proof of Theorem \ref{thmsz}]
Combine Proposition \ref{prop:derivation pb 1 boule} and Proposition \ref{prop:pb 1 boule}.
\end{proof}

To prove Proposition \ref{prop:pb 1 boule}, one shall consider without loss of generality that $B$ is the unit ball. Then, any minimiser $v$, being radially symmetric, can be thought as a function of the one-dimensional radial variable $r\in[0,1)$. Furthermore, $v$ belongs to $C^1(\overline{B}\setminus\{0\})$ (this follows from the Radial Lemma in \cite{strauss} applied to $v(r)-(r-1)\partial_r v(1)$, which is in $H_0^2(B)\subseteq H^2(\R^d)$). As a first step, let us prove that $v$ changes sign in at most one point $\rho\in(0,1)$.

\begin{lemme}\label{lemme:un seul zero}
Let $v$ be a minimiser of \eqref{eq:pb 1 boule} such that $\partial_nv>0$. Then, there exists $\rho\in(0,1)$ such that $v$ is non-negative in $[0,\rho]$ and negative in $(\rho,1)$.
\end{lemme}

\begin{proof}
First, $v$ admits zeros in $(0,1)$, otherwise it would be non-negative (since $v\mygeq0$), contradicting $\partial_nv>0$. Let $0<\rho<1$ be the greatest of those zeros. It does exist, otherwise there would be a sequence of zeros converging to $1$. As a result, we would have $\partial_n v=0$ (here we use that $v\in C^1(\overline{B}\setminus\{0\})$, as justified above), contradicting once again the assumption. So the sequence of zeros converges to some radius $\rho\in(0,1)$, which is itself a zero by continuity of $v$. So the greatest zero $\rho$ exists, and $v<0$ in $(\rho,1)$.

We now consider the continuous function $w$ such that $-\partial_r w=|\partial_r v|$ over $[0,\rho)$ and $w=v$ over $[\rho,1)$ (see Figure \ref{fig:argument nadirashvili}). The function $w$ then satisfies $|\partial_r w|=|\partial_r v|$, $|\Delta v|=|\Delta w|$, and $|w|\geq|v|$ (for $r<\rho$, this follows from $w(r)=-\int_r^\rho\partial_r w\geq\left|-\int_r^\rho \partial_r v\right|=|v(r)|$) with equality if and only if  $v$ is monotonous in $[0,\rho)$. As a result $w$ is a competitor which makes the quotient decrease (a contradiction) except if $v$ is monotonous in $[0,\rho)$. In this last situation, $v$ is one-sign over $[0,\rho)$, and it must be non-negative, otherwise the assumption $v\mygeq0$ would fail.
\end{proof}

\begin{figure}
\includegraphics{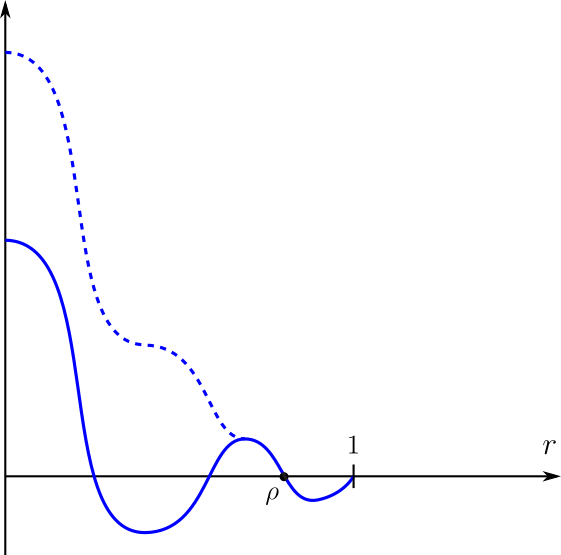}
\caption{Radial profiles of the minimiser $v$ of the incomplete one-ball problem (solid line), and the competitor $w$ constructed in the proof of Lemma \ref{lemme:un seul zero} (dotted line).}
\label{fig:argument nadirashvili}
\end{figure}

\begin{remarque}\label{rmq:nadirashvili}
The construction of the competitor $w$ in the above proof was inspired by \cite[equation (3.12)]{nadirashvili} (see also \cite[Remark 3.2]{ashbaugh-bucur-laugesen-leylekian}).
\end{remarque}

As a consequence of Lemma \ref{lemme:un seul zero}, the function $r\mapsto \int_{B_r}v$ is non-decreasing in $(0,\rho)$ and decreasing in $(\rho,1)$. Since it is non-negative (by assumption), this means that it is positive in $(\eta,1)$ for some $\eta\in[0,\rho)$. With this observation in mind, one may prove Proposition \ref{prop:pb 1 boule}.

\begin{proof}[Proof of Proposition \ref{prop:pb 1 boule}]
We assume by contradiction that $\partial_n v>0$. To begin, take $h\in H^2(B)$ radially symmetric with support in $(\eta, 1)$, and such that $\int_B h=0$. By definition of $\eta$ (see above), $v+th$ is an admissible function for any small enough $t\in\R$ (it is central here to allow $t$ to be negative). The first variation of the quotient then shows that
$$
\Delta^2 v-\tau\Delta v=\Gamma v+K\quad\text{in }B\setminus\overline{B_\eta},
$$
where $\Gamma$ is the (non-negative) value of the minimum in \eqref{eq:pb 1 boule}, and $K$ is some Lagrange multiplier coming from the constraint $\int_B h=0$. For the boundary conditions, take now a function $h$ with support in $(\eta,1]$ (and still with zero mean value). The normal derivative of $h$ may not be zero, but $v+th$ is still admissible for small $t\in\R$ since $\partial_nv>0$. The boundary term in the first variation now yields
$$
\Delta v =0\text{ on }\partial B.
$$
But we can actually do the same with a non-negative function $h\in H_0^2(B)$. Indeed, for any small $t>0$ (and here it is central that $t$ is positive), the function $v+th$ is still admissible. In this case, the first variation gives
$$
\Delta^2 v-\tau\Delta v-\Gamma v\geq0\text{ in }B.
$$
Now recall the factorisation $\Delta^2-\tau\Delta-\Gamma=(\Delta-\alpha_+)(\Delta-\alpha_-)$, where
$$
\alpha_+:=\tau/2+\sqrt{\Gamma+\tau^2/4}\geq 0,\qquad\alpha_-:=\tau/2-\sqrt{\Gamma+\tau^2/4}\leq0.
$$ 
So if we define $w:=(\Delta-\alpha_-)v$, the function $w$ satisfies
$$
\left\{
\begin{array}{rcll}
(\Delta-\alpha_+) w & \geq & 0 & \text{in }B,\\
w & = & 0 & \text{on }\partial B.
\end{array}
\right.
$$
The boundary condition for $w$ comes from the fact that $v=\Delta v=0$. By the maximum principle (recall that $\alpha_+\geq0$), we deduce that $w\leq0$ in $B$. Since $\alpha_-\leq0$ and $\int_B v\geq0$, the result is unequivocal:
$$
0\geq\int_Bw=\int_{B}\Delta v-\alpha_-\int_B v\geq \int_{\partial B}\partial_n v.
$$
This contradicts the assumption $\partial_n v>0$. Therefore, we must have $\partial_n v=0$ and $v\in H_0^2(B)$.
\end{proof}

\section{Saint--Venant-type result}\label{sec:torsion}

In this section, we consider the torsional rigidity for $S_{\tau}$, focused towards the proof of the Saint-Venant result given by Theorem~\ref{thmsv}. The first step of the proof (§~\ref{subsec:pb 2 boules complet}) consists in the derivation of a two-ball problem, and the completion of the boundary coupling condition. Then, the sensitivity of the two-ball problem with respect to the relative size of the two balls is investigated in §~\ref{subsec:pb 2 boules}. Putting all together provides the proof of Theorem~\ref{thmsv} in §~\ref{subsec:saint-venant}. To conclude, we study the torsion function and the torsional rigidity in the case of balls (§~\ref{subsec:torsion}). This complements Theorem~\ref{thmsv}, by making the bound on the torsional rigidity explicit.

\subsection{Derivation and completion of the two-ball problem}\label{subsec:pb 2 boules complet}

As explained in the introduction, the two-ball problem which is first obtained contains an incomplete coupling condition. This follows from a hybrid symmetrisation procedure applied in the positive and the negative nodal domains, as shown in the next proposition.

\begin{proposition}\label{prop:derivation pb 2 boules}
Let $\Omega$ be an open subset of $\R^d$ of finite measure and let $\tau\geq0$. Let $W$ be the torsion function \eqref{torsion-eq} over $\Omega$, and consider the balls $B_a=\{W>0\}^*$ and $B_b=\{W<0\}^*$. Then,
$$
-T(\Omega,\tau)\geq \min_{\substack{(v,w)\in\mathcal{H}_{a,b}\\\int_{\partial B_a}\partial_n v\geq\int_{\partial B_b}\partial_n w}}\int_{B_a}(\Delta v)^2+\tau\int_{B_a}|\nabla v|^2-2\int_{B_a} v+\int_{B_b}(\Delta w)^2=:E_{inc}(a,b,\tau),
$$
where $\mathcal{H}_{a,b}$ denotes the space of pairs $(v,w)$ of radial functions over $B_a$ and $B_b$ respectively, vanishing at the boundary, and which are more concentrated than zero. That is,
$$
\mathcal{H}_{a,b}:=\{(v,w)\text{, }v\in H_0^1\cap H^2(B_a)\text{ and }w\in H_0^1\cap H^2(B_b)\text{ radial functions, }v,w\mygeq0\}.
$$
\end{proposition}

\begin{remarque}\label{rmq:Ha,b}
Let us agree that when $a=0$, $(v,w)\in\mathcal{H}_{a,b}$ means that $v\equiv0$ and $\partial_n v=0$, whereas when $b=0$ it means that $w\equiv0$ and $\partial_n w=0$. With this convention, the above result holds even in these degenerate cases.
\end{remarque}

\begin{remarque}
An inspection of the proof shows that one can actually take $w\geq0$ instead of $w\mygeq0$. But since this does not really help in the remaining of the section (see however Lemma~\ref{lemme:pas de zero}), we only mention it as a remark.
\end{remarque}

\begin{proof}
Recall that $T(\Omega,\tau)$ admits the variational formulation \eqref{eq:caracterisation variationnelle torsion}, with the torsion function $W$ as a minimiser. Let us denote
$W_+=W|_{\Omega_+}$ and $W_-=-W|_{\Omega_-}$ the positive and negative parts of $W$. As already explained, the idea is now to factorise (only!) the positive part in the energy, in
order to display the term $(\Delta-\tau/2)W_+$:
\begin{align}\label{eq:factorisation rigidité torsionnelle}
\begin{split}
-T(\Omega,\tau) = & \int_{\Omega_+}\left(\Delta W_+-\frac{\tau}{2}W_+\right)^2-\tau^2/4\int_{\Omega_+} W_+^2-2\int_{\Omega_+} W_+\\
& \hspace*{5mm} + \int_{\Omega_-}(\Delta W_-)^2+\tau\int_{\Omega_-}|\nabla W_-|^2+2\int_{\Omega_-} W_-.
\end{split}
\end{align}
As usual, $\Omega_\pm=\{\pm W>0\}$ denote the nodal domains of $W$. Now, write $f=-\Delta W_++\frac{\tau}{2}W_+$ and $g=-\Delta W_-$, and consider the competitors $v$ and $w$ defined by
\begin{equation*}
\left\{
\begin{array}{rcll}
-\Delta v+\frac{\tau}{2}v & = & f^* & \text{in }B_a,\\
v & = & 0 & \text{on }\partial B_a,
\end{array}
\right.
\qquad\text{and}\qquad
\left\{
\begin{array}{rcll}
-\Delta w & = & g^* & \text{in }B_b,\\
w & = & 0 & \text{on }\partial B_b,
\end{array}
\right.
\end{equation*}
where $B_a=\Omega_+^*$ and $B_b=\Omega_-^*$. This symmetrisation procedure is said to be hybrid, for it involves different operators in the different nodal domains. We now apply Talenti's comparison principle (Theorem~\ref{thm:talenti}) to compare $W_+$ and $v$, which gives:
\begin{align*}
\int_{\Omega_+}W_+\leq\int_{B_a} v,\quad\int_{\Omega_+}W_+^2\leq\int_{B_a} v^2.
\end{align*}
This allows to bound from below the terms in $W_+$ in \eqref{eq:factorisation rigidité torsionnelle}. On the other hand, applying the standard version of Talenti's comparison principle \cite[Theorem 3.1.1]{kesavan}, we obtain the a.e. pointwise comparison $W_-^*\leq w$. Unfortunately, this does not allow us to address the terms with lower order derivatives of $W_-$ in \eqref{eq:factorisation rigidité torsionnelle}, and these lower order terms must simply be discarded. Putting everything together, we obtain
$$
-T(\Omega,\tau)\geq\int_{B_a}(\Delta v)^2+\tau\int_{B_a}|\nabla v|^2-2\int_{B_a} v+\int_{B_b}(\Delta w)^2.
$$
Now it remains to check that the pair $(v,w)$ is admissible for $E_{inc}(a,b,\tau)$. First, by elliptic regularity, we have $v\in H_0^1\cap H^2(B_a)$ and $w\in H_0^1\cap H^2(B_b)$. Second, $v\mygeq W_+^*\geq0$ and $w\geq W_-^*\geq0$. Lastly, by equimeasurability of the Schwarz symmetrisation, we have
$$
\int_{B_a}\left(-\Delta v+\frac{\tau}{2}v\right)-\int_{B_b}(-\Delta w)=-\int_\Omega\Delta W+\frac{\tau}{2}\int_{\Omega_+}W_+,
$$
or equivalently
$$
\int_{\partial B_a}\partial_n v-\int_{\partial B_b}\partial_n w=\int_{\partial\Omega}\partial_nW+\frac{\tau}{2} \left(\int_{B_a}v-\int_{\Omega_+}W_+\right).
$$
Since $\partial_nW=0$ and since $v\mygeq W_+^*$, the right-hand side is non-negative. This shows that $(v,w)$ is admissible as expected. To conclude, let us mention that the minimisation problem is well-posed. This follows from the fact that $\{(v,w)\in\mathcal{H}_{a,b}:\int_{B_a}\Delta v\geq\int_{B_b}\Delta w\}$ is weakly closed in $H^2(B_a)\times H^2(B_b)$.
\end{proof}

Proposition \ref{prop:derivation pb 2 boules} gives a lower bound on (minus) the torsional rigidity in terms of a two-ball problem $E_{inc}(a,b,\tau)$ with the incomplete boundary coupling condition
$$
\int_{\partial B_a}\partial_n v\geq\int_{B_b}\partial_n w,
$$
for an admissible pair $(v,w)\in \mathcal{H}_{a,b}$. Our goal is now to show that $E_{inc}(a,b,\tau)$ is in some sense equivalent to the two-ball problem with completed coupling condition
\begin{equation}\label{eq:pb 2 boules complet}
E(a,b,\tau)=\min_{\substack{(v,w)\in\mathcal{H}_{a,b}\\\int_{\partial B_a}\partial_n v=\int_{\partial B_b}\partial_n w}}\int_{B_a}(\Delta v)^2+\tau\int_{B_a}|\nabla v|^2-2\int_{B_a} v+\int_{B_b}(\Delta w)^2.
\end{equation}
Actually, the next result shows that $E_{inc}(a,b,\tau)$ can be bounded from below  by the smallest of either  $E(a,b,\tau)$ or the torsional rigidity over $B_a$.

\begin{proposition}\label{prop:derivation pb 2 boule complet}
Let $a,b\geq0$. Then,
$$
E_{inc}(a,b,\tau)\geq\min\left\{E(a,b,\tau),-T(B_a,\tau)\right\}.
$$
\end{proposition}

\begin{remarque}\label{rmq:torsion d'une boule vide}
Note that when $a=0$ or $b=0$, we define~\eqref{eq:pb 2 boules complet} according to the convention of Remark~\ref{rmq:Ha,b}. Similarly,
we make the convention that $T(B_a,\tau)=0$ when $a=0$, which makes the proposition hold even in this case.
\end{remarque}

To prove this proposition, we found more convenient to distinguish two different situations, depending on the minimisers of $E_{inc}(a,b,\tau)$. Indeed, for a minimising pair $(v,w)\in\mathcal{H}_{a,b}$, we will address the cases $\partial_n v\geq0$ and $\partial_nv<0$ separately. The case $\partial_n v\geq0$ appears to be simpler. In this situation, it is clear that $w\equiv0$ is admissible, and that this choice minimises the energy. Therefore we have
$$
E_{inc}(a,b,\tau)\geq\min_{\substack{v\in H_0^1\cap H^2(B_a),\\v\text{ radial,}\\v\mygeq0\text{, }\partial_nv\geq0}}\int_{B_a}(\Delta v)^2+\tau\int_{B_a}|\nabla v|^2-2\int_{B_a} v.
$$
The right-hand side is a problem which is very similar to that addressed in Proposition \ref{prop:pb 1 boule}, and it can be treated accordingly. For the case $\partial_n v<0$, the constraint yields $\partial_n w<0$. Thanks to these conditions, we will prove that both $v$ and $w$ are positive, and this will ultimately allow us to conclude that $\int_{\partial B_a}\partial_n v=\int_{\partial B_b}\partial_n w$.

\subsection*{Case 1: $\partial_n v\geq0$}\label{subsec:cas 1}

As explained above, in this situation we assume $w$ to vanish identically, and the problem boils down to the study of
\begin{equation}\label{eq:rigidité torsionnelle incomplete}
T_{inc}:=\min_{\substack{v\in H_0^1\cap H^2(B_a),\\v\text{ radial,}\\v\mygeq0\text{, }\partial_nv\geq0}}\int_{B_a}(\Delta v)^2+\tau\int_{B_a}|\nabla v|^2-2\int_{B_a} v,
\end{equation}
with the convention $T_{inc}=0$ whenever $a=0$. Thus we focus on the case $a>0$. First, copying word by word the proof of Lemma \ref{lemme:un seul zero} (the only thing to change is to replace the radius $1$ by the radius $a$), we obtain the following:

\begin{lemme}\label{lemme:un seul zero torsion}
Let $v$ be a minimiser of $T_{inc}$ such that $\partial_n v>0$. Then, there exists $\rho\in(0,a)$ such that $v$ is non-negative in $[0,\rho]$ and negative in $(\rho,a)$.
\end{lemme}

As a consequence, we notice here again that $\int_{B_r} v$ is positive for all $r\in(\eta,a)$ for some $\eta\in[0,\rho)$. This observation allows to prove the analogue of Proposition \ref{prop:pb 1 boule}:

\begin{proposition}\label{prop:v'>0}
If $a>0$, any minimiser of $T_{inc}$ is in $H_0^2(B_a)$.
\end{proposition}

\begin{proof}
Let us sketch the argument, and refer to the proof of Proposition~\ref{prop:pb 1 boule} for more details. We assume by contradiction that a minimiser $v$ satisfies $\partial_n v>0$. By Lemma \ref{lemme:un seul zero torsion} and the subsequent remark, we shall use a radial test function supported in $(\eta,a)$ with zero mean value to perturb $v$. In this way, we get
$$
\Delta^2v-\tau\Delta v=constant\text{ in }B_a\setminus\overline{B_\eta}.
$$
Taking then a test function with support in $(\eta,a]$, we end up with the boundary condition
$$
\Delta v=0\text{ on }\partial B_a.
$$
Lastly, considering a non-negative radial test function in $H_0^2(B_a)$, we also have
$$
\Delta^2 v-\tau\Delta v\geq 1.
$$
Combine the new boundary condition $\Delta v=0$ and the maximum principle for the operator $-\Delta+\tau$ to conclude that $-\Delta v\geq0$. This contradicts the assumption $\partial_n v>0$.
\end{proof}

\subsection*{Case 2: $\partial_n v<0$}\label{subsec:cas 2}

In this situation, the first step is, in the wake of Lemma \ref{lemme:un seul zero}, to prove that we can assume $v$ and $w$ to be positive, up to changing the minimising pair. This is performed in the following lemma. Note that, due to our convention, $a$ and $b$ must here both be positive.

\begin{lemme}\label{lemme:pas de zero}
Let $(v,w)$ be a minimising pair for $E_{inc}(a,b,\tau)$, and assume that $\partial_n v$ and $\partial_n w$ are negative. Then, there exist positive minimisers $\tilde{v}$ and $\tilde{w}$
of $E_{inc}(a,b,\tau)$ with the same boundary conditions.
\end{lemme}

\begin{proof}
Consider the radial function $\tilde{v}$ such that $\tilde{v}=0$ on $\partial B_a$ and $\partial_r\tilde{v}=-|\partial_r v|$. This function is admissible and does not increase the energy (see
the proof of Lemma~\ref{lemme:un seul zero} for more details). Furthermore, $\partial_n\tilde{v}=-|\partial_n v|=\partial_n v<0$. Lastly, $\tilde{v}$ is non-increasing and has negative normal
derivative at the boundary, hence it is positive. A similar construction for $w$ provides $\tilde{w}$.
\end{proof}

With this new information, we shall complete the coupling condition in $E_{inc}(a,b,\tau)$ and prove the following result.

\begin{proposition}\label{prop:v'<0}
Let $(v,w)$ be a minimising pair for $E_{inc}(a,b,\tau)$, and assume $\partial_n v$ to be negative. Then,
$$
\int_{\partial B_a}\partial_n v=\int_{\partial B_b}\partial_n w.
$$
\end{proposition}

\begin{proof}
Suppose by contradiction that $\int_{\partial B_a}\partial_n v>\int_{\partial B_b}\partial_n w$. Since we assumed $\partial_n v<0$, we also get $\partial_n w<0$. By Lemma \ref{lemme:pas de zero},
up to changing the minimising pair, we shall consider that $w>0$. Note that we may also assume that $v>0$ but this will be irrelevant. Indeed, we now use an arbitrary radial function $h\in
H_0^1\cap H^2(B_b)$ to perturb the minimiser $w$. Since $w>0$, for any small enough $t\in\R$, the function $w+th$ is still admissible. As a result (see again the proof of Proposition~\ref{prop:pb 1 boule} for more details), we obtain that
$$
\Delta^2 w=0\text{ in }B_b,
$$
with the additional boundary condition
$$
\Delta w=0\text{ on }\partial B_b.
$$
This yields $w\equiv0$, contradicting $\partial_n w<0$.
\end{proof}

\subsection*{Completion of the coupling condition}

Combining Propositions~\ref{prop:v'>0} and~\ref{prop:v'<0}, we can complete the coupling condition in problem $E_{inc}(a,b,\tau)$, which amounts to proving Proposition \ref{prop:derivation pb 2 boule complet}.

\begin{proof}[Proof of Proposition \ref{prop:derivation pb 2 boule complet}]
Let $(v,w)$ be a minimising pair for $E_{inc}(a,b,\tau)$. If $\partial_n v\geq0$, we are in \hyperref[subsec:cas 1]{case 1}, hence $E_{inc}(a,b,\tau)\geq T_{inc}=-T(B_a,\tau)$, the last equality following from Proposition~\ref{prop:v'>0}, when $a>0$ (and from the convention $T(B_a,\tau)=T_{inc}=0$ when $a=0$). If on the contrary $\partial_nv<0$, we fall in \hyperref[subsec:cas 2]{case 2}, and by virtue of Proposition \ref{prop:v'<0}, we obtain that $E_{inc}(a,b,\tau)= E(a,b,\tau)$. Overall,
$$
E_{inc}(a,b,\tau)\geq\min(-T(B_a,\tau),E(a,b,\tau)).
$$
\end{proof}

\subsection{Resolution of the two-ball problem}\label{subsec:pb 2 boules}

The last step in the proof of Theorem~\ref{thmsv} concerns the analysis of the two-ball problem with completed coupling condition $E(a,b,\tau)$ defined in~\eqref{eq:pb 2 boules complet}. The goal is to show that, under the constraint $|B_a|+|B_b|=constant$, the energy of the two-ball problem is minimal when $b=0$, for which one retrieves the torsional rigidity $T(B_a,\tau)$. As usual, it is enough to restrict to the case $|B_a|+|B_b|=|B_1|$, or in other words $a^d+b^d=1$. In this setting, we show that $E(a,(1-a^d)^{1/d},\tau)$ decreases with respect to $a\in(0,1)$. This is the purpose of the next proposition (see also Figure \ref{fig:pb2boules torsion}).

\begin{proposition}\label{prop:pb 2 boules torsion}
Let $\tau\geq0$ and introduce the notation $E(a,\tau):=E(a,(1-a^d)^{1/d},\tau)$ where $E(a,b,\tau)$ is defined in~\eqref{eq:pb 2 boules complet}. Then, for all positive values of $a$ and $b$
we have
\begin{subnumcases}{\label{eq:pb 2 boules torsion explicite, tout tau} E(a,b,\tau)=}
-\frac{|B_a|}{d\tau^2}\left[\frac{\left(d-a\sqrt{\tau}\frac{I_\nu}{I_{\nu+1}}(a\sqrt{\tau})\right)^2}{d\frac{a^d}{b^d}+a\sqrt{\tau}\frac{I_\nu}{I_{\nu+1}}(a\sqrt{\tau})}+\left(d-a\sqrt{\tau}\frac{I_\nu}{I_{\nu+1}}(a\sqrt{\tau})+\frac{a^2\tau}{d+2}\right)\right] & if $\tau>0$,\label{eq:pb 2 boules torsion explicite}\\
-\frac{a^4|B_a|}{d(d+2)^2}\left(\fr{1}{d+4}+\fr{1}{d(1+a^d/b^d)}\right) & if $\tau=0$.\label{eq:pb 2 boules torsion explicite tau=0}
\end{subnumcases}
Furthermore, for all $a\in(0,1)$,
$$
\partial_aE(a,\tau)<0
$$
and hence
$$
E(a,\tau)> E(1,\tau)=-T(B_1,\tau).
$$
\end{proposition}

\begin{remarque}
Observe that when $a\to0$ in \eqref{eq:pb 2 boules torsion explicite, tout tau}, we have $E(a,b,\tau)\to0$, while when $b\to0$ we obtain (see \eqref{eq:rigidité torsionnelle de la boule} and \eqref{eq:rigidité torsionnelle de la boule tau=0} below),
$$
E(a,b,\tau)\to
\left.
\begin{cases*}
-\frac{|B_a|}{d\tau^2}\left(d-a\sqrt{\tau}\frac{I_\nu}{I_{\nu+1}}(a\sqrt{\tau})+\frac{a^2\tau}{d+2}\right), & for $\tau>0$\\
-\frac{a^4|B_a|}{d(d+2)^2(d+4)}, & for $\tau=0$
\end{cases*}
\right\}=-T(B_a,\tau)
$$
\end{remarque}

\begin{figure}[htbp]
    \centering
    \begin{subfigure}[t]{0.9\textwidth}
        \centering
        \includegraphics[width=\textwidth]{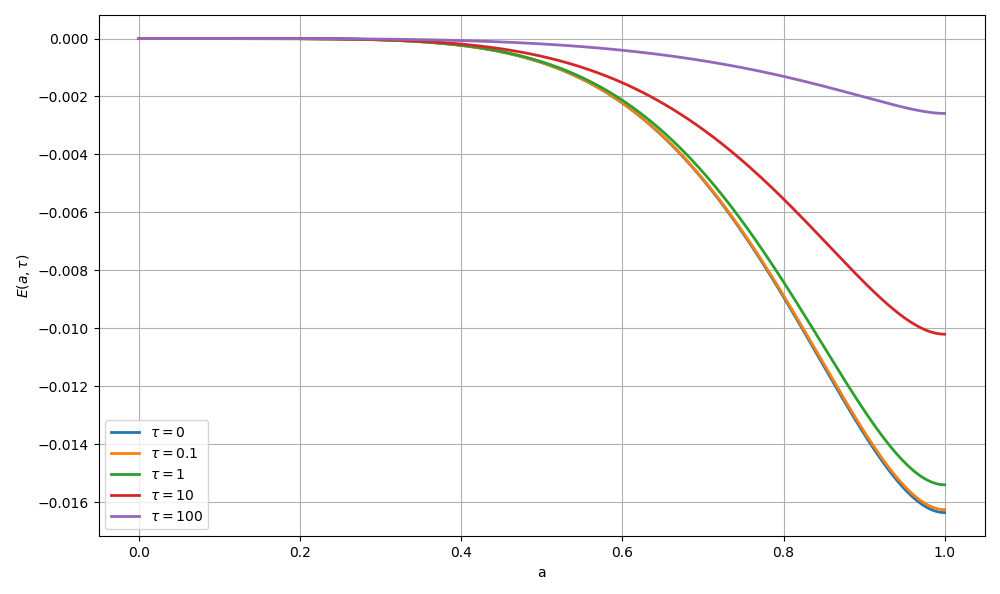}
        \caption{Dimension $d=2$}
        \label{fig:E_d2}
    \end{subfigure}
    
    \vspace{0.5cm}  

    \begin{subfigure}[t]{0.9\textwidth}
        \centering
        \includegraphics[width=\textwidth]{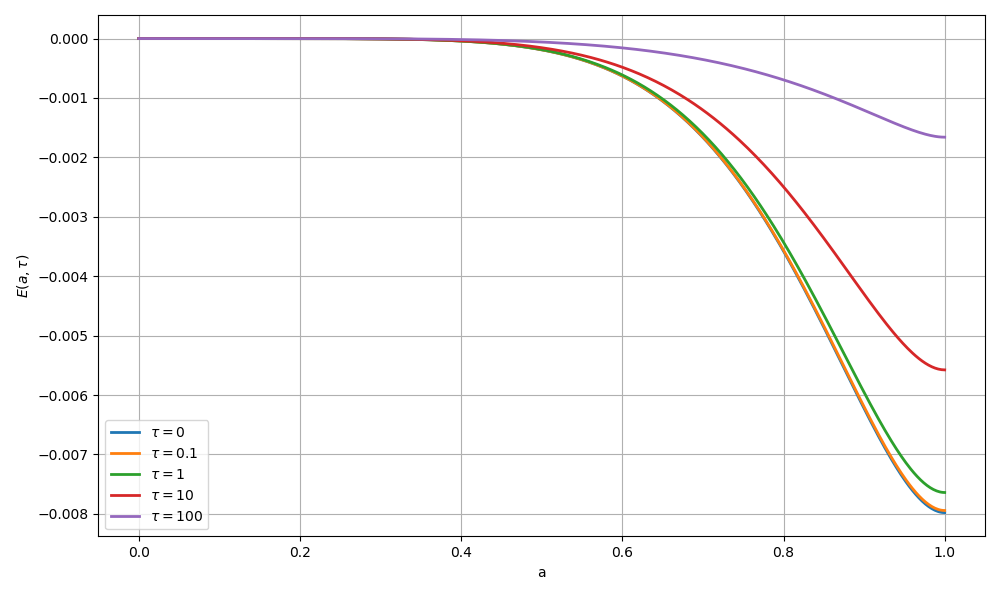}
        \caption{Dimension $d=3$}
        \label{fig:E_d3}
    \end{subfigure}

    \caption{Plot of the two-ball energy $E(a,\tau)$ given in~\eqref{eq:pb 2 boules torsion explicite, tout tau} with respect to $a$, for different values of $\tau$ and different dimensions $d$. As stated in Proposition \ref{prop:pb 2 boules torsion}, the curves are all decreasing.}
    \label{fig:pb2boules torsion}
\end{figure}

The proof is essentially based on tedious yet meticulous computations that we gathered in Appendix~\ref{annexe:resolution pb 2 boules torsion}. As a result, combining Proposition~\ref{prop:derivation pb 2 boules}, Proposition~\ref{prop:derivation pb 2 boule complet} and Proposition~\ref{prop:pb 2 boules torsion}, we finally end up with the proof of Theorem~\ref{thmsv}.

\subsection{Proof of Theorem~\ref{thmsv}}
\label{subsec:saint-venant}

By Proposition~\ref{prop:derivation pb 2 boules}, we have
$$
-T(\Omega,\tau)\geq E_{inc}(a,b,\tau),
$$
with balls $B_a$ and $B_b$ such that $|B_a|+|B_b|=|\Omega|$. The last point follows from the fact that the torsion function over $\Omega$ does not vanish over a set of positive measure \cite{mityagin} by analyticity \cite[Chapter V, Section 4, Theorem 1]{dautray-lions}. But by Proposition \ref{prop:pb 1 boule}, we have
$$
E_{inc}(a,b,\tau)\geq\min(E(a,b,\tau),-T(B_a,\tau)).
$$
Note that $-T(B_a,\tau)\geq -T(\Omega^*,\tau)$, since any test function over $B_a$ can be used as a test function over $\Omega^*$ in \eqref{eq:caracterisation variationnelle torsion}. On the other hand, by Proposition \ref{prop:pb 2 boules torsion}, we also obtain $E(a,b,\tau)\geq -T(\Omega^*,\tau)$. Observe that, in order to use Proposition \ref{prop:pb 2 boules torsion}, one possibly needs to rescale the problem to have $|\Omega|=|B_1|$, and this changes the parameter $\tau$. But this is not an issue insofar as that proposition holds for any $\tau\geq0$. As a result, we obtain as desired
$$
-T(\Omega,\tau)\geq -T(\Omega^*,\tau).
$$

\subsection{Torsion function in the ball\label{torsionball}}\label{subsec:torsion}

To complement the Saint--Venant-type result of Theorem~\ref{thmsv}, we will now provide an explicit expression for the torsion function and the torsional rigidity over balls. Let us point out that here we do not restrict to the case of a non-negative tension parameter. In this setting, the existence and the uniqueness of the torsion function is subject to the buckling problem. Recall that, as shown in \cite{coster-nicaise-troestler15}, the buckling eigenvalues $\Lambda$ and the associated eigenfunctions $u$ over the ball of radius $R$ are of the form
\begin{equation}\label{buckleig}
\Lambda = \fr{j_{\nu+\kappa+1,i}^2}{R^2},\qquad u(r,\theta) = \left[ r^{-\nu}J_{\nu+\kappa}(\sqrt{\Lambda}r) - R^{-\nu-\kappa}J_{\nu+\kappa}(\sqrt{\Lambda}R)r^{\kappa}\right]
S_{\kappa}(\theta),
\end{equation}
for $\kappa\in\N$ and $i\in\N^*$, where $S_\kappa$ denotes the spherical harmonic of order $\kappa$.

\begin{theoreme}\label{thm:torsion boule} When $B\subset \R^{d}$ $(d\geq 2)$ is a ball of radius $R$, the torsion function~\eqref{torsion-eq} exists
and is unique if and only if $-\tau$ is not an eigenvalue of the buckling problem~\eqref{eq:pb vp buckling}. In this case, the solution is radial and may be written explicitly as
\begin{equation}\label{eq:torsion boule}
w(r) =
\left\{
\begin{array}{ll}
\fr{R^2 - r^2}{2  d\tau} + \fr{R}{d \tau^{3/2}}
\left[\fr{I_{\nu} \left(\sqrt{\tau} r \right)}{I_{\nu+1} \left(\sqrt{\tau} R \right)}\left(\fr{r}{R}\right)^{-\nu} -
\fr{I_{\nu}\left(\sqrt {\tau}  R \right)}{I_{\nu+1}\left(\sqrt {\tau}  R \right)} \right], & \tau \neq 0\eqskip
\fr{\left(R^2-r^2\right)^{2}}{8d(d+2)}, & \tau = 0.
\end{array}
\right.
\end{equation}
If $-\tau$ is a buckling eigenvalue associated with a radial eigenfunction, then~\eqref{torsion-eq} has
no solution while if a corresponding eigenfunction $v$ is not radial, then there exists an infinite number of
solutions of the form
\[
w(r,\theta) = \frac{R^2 - r^2}{2  d\tau} + \fr{R}{d \tau^{3/2}}
\left[\fr{I_{\nu} \left(\sqrt{\tau} r \right)}{I_{\nu+1} \left(\sqrt{\tau} R \right)}\left(\fr{r}{R}\right)^{-\nu} -
\fr{I_{\nu}\left(\sqrt {\tau}  R \right)}{I_{\nu+1}\left(\sqrt {\tau}  R \right)} \right] +  C v(r,\theta),\qquad C\in\R.
\]
\end{theoreme}

\begin{remarque}
When $\tau$ is negative, one should understand $\sqrt{\tau}$ as $i\sqrt{-\tau}$. As a result, one can replace the Bessel functions of type $I$ by Bessel functions of type $J$ and $\tau$ by $-\tau$ in formula~\eqref{eq:torsion boule}.
\end{remarque}

\begin{remarque}
 We note that the (pointwise) limit of the solution for nonzero $\tau$ as $\tau$ goes to zero is the polynomial
 solution for $\tau$ equal to zero.
\end{remarque}

\begin{remarque}
One can check that the cases where $-\tau$ is associated with radial buckling eigenfunctions and the cases where $-\tau$ is associated with non-radial buckling eigenfunctions are mutually exclusive. This is because $\{j_{\nu+1,i}\}_{i\in\N^*}$ is disjoint from $\{j_{\nu+\kappa+1,i}\}_{i\in\N^*}$ for $\kappa>0$.
\end{remarque}

\begin{remarque}
For positive values of $\tau$ the torsional rigidity of the ball may be evaluated explicitly and is given by
\begin{equation}\label{eq:rigidité torsionnelle de la boule}
 T(B,\tau)=\dint_{B} w\ {\rm d}V = \fr{|B|R^d}{d \tau^2}\left[ d + \fr{R^2 \tau}{d+2}- R\sqrt{\tau} \fr{I_{\nu}\left(R\sqrt{\tau}\right)}{I_{\nu+1}\left(R\sqrt{\tau}\right)}\right].
\end{equation}
This expression also holds for negative values of $\tau$ such that $-\tau$ is not an eigenvalue of the buckling
problem~\eqref{buckleig}, and (thanks to~\eqref{eq:expansion ratio Bessel} in appendix) the limit value when $\tau$ converges to zero coincides with the value of the
torsional rigidity at that point, namely
\begin{equation}\label{eq:rigidité torsionnelle de la boule tau=0}
 T(B,0) = \fr{|B|R^4}{d(d+4)(d+2)^2}.
\end{equation}
Lastly, when $-\tau$ is a buckling eigenvalue associated with a non-radial eigenfunction, the torsion functions all have the same integral, which is given by~\eqref{eq:rigidité torsionnelle de la boule}, hence the torsional rigidity can be extended to these values of $\tau$. Whereas when $-\tau$ is a buckling eigenvalue associated with a radial eigenfunction (that is $-\tau=j_{\nu+1,i}^2/R^2$, for some $i\in\N^*$), expression \eqref{eq:rigidité torsionnelle de la boule} diverges, which is compatible with the fact that no torsion function exists.
\end{remarque}

\begin{remarque}\label{rmq:positivité torsion}
From its expression, it is unclear wether the torsion function $w$ remains positive or not. For $\tau\geq-j_{\nu+1}^2/R^2$, this holds by constructing a competitor in the same way as in the proof of Lemma~\ref{lemme:un seul zero} (see also \cite{laurençot-walker} for $\tau\geq0$). But this cannot be expected for all values of $\tau$, as the integral of $w$ goes to $-\infty$ when $\tau$ approaches $-j_{\nu+1,i}^2/R^2$ from below, for any $i\in\N^*$.
\end{remarque}

\begin{proof}
We note that the first term in $w$, namely $w_{0}(r)=(R^2 - r^2)/(2  d\tau)$ is the torsion function for
the operator $-\tau \Delta $ that is, it satisfies the equation $-\tau \Delta w_0 = 1$ and vanishes on the
boundary of $B$. Since $\Delta^{2}w_{0}=0$, we see that it is also a solution of
$\left(\Delta^{2}-\tau\Delta \right)u=1$, except that it does not satisfy the boundary condition on the derivative.
However, we may add to this function any other radial function which is zero on the boundary as long as it satisfies
the equation $\left(\Delta^{2}-\tau\Delta \right)u=0$ in $B$, namely, any multiple of
\[
 r^{-\nu}  I_{\nu} \left(\sqrt{\tau} r \right) -R^{-\nu}
I_{\nu}\left(\sqrt {\tau}  R \right).
\]
When $-\tau$ is not an buckling eigenvalue and is not zero, by equating the normal derivative of the resulting function on the boundary to zero, we obtain the first expression
for the function $w$ above. In the case where $\tau$ vanishes, it is easy to integrate the equation $(r^{1-d}\partial_rr^{d-1}\partial_r)^2 w=1$ to see that the given polynomial is
the desired solution (see also \cite{bennett}). Since when $-\tau$ is not an eigenvalue of the buckling problem, the solution of~\eqref{torsion-eq}
is unique, the result follows in this case.

If $-\tau$ is an eigenvalue of the buckling problem, by the Fredholm alternative there are two possible situations:
either the corresponding eigenfunctions are not orthogonal (in the sense of $L^{2}(B)$) to the contant function,
in which case there is no solution to equation~\eqref{torsion-eq}, or they do have zero average and solutions exist
but they are not unique.

In the case of the radial eigenfunctions given by~\eqref{buckleig} with $\kappa=0$ we have
\[
\begin{array}{lll}
 \dint_{B_{R}} v(r) \ {\rm d}x & = & \left|\mathbb{S}^{d-1}\right| \dint_{0}^{R} r^{d-1} v(r)  {\rm d}r\eqskip
 & = & \left|\mathbb{S}^{d-1}\right| \dint_{0}^{R} r^{d-1}\left[
 r^{1-d/2}J_{d/2-1}(\sqrt{-\tau}r) - R^{1-d/2}J_{d/2-1}(\sqrt{-\tau}R)\right] {\rm d}r\eqskip
 & = & \left|\mathbb{S}^{d-1}\right|\fr{R^{1+d/2}}{d} J_{d/2+1}\left(j_{d/2,i}\right).
\end{array}
\]
Since the zeros of Bessel functions of different orders never coincide, the above quantity is never zero and
there are no solutions of~\eqref{torsion-eq} in this case.

When $v$ has a nontrivial angular part, and since all spherical harmonics with positive index $\kappa$ are orthogonal to the constant
function on $\mathbb{S}^{d-1}$, the integral over the sphere vanishes, and the corresponding solutions
have zero average. As a consequence, solutions of~\eqref{torsion-eq} are not unique. Furthermore, a particular solution of~\eqref{torsion-eq} is given by the first line of~\eqref{eq:torsion boule}, hence all the solutions are the sum of this particular solution and of a solution of the homogeneous equation, which is exactly the buckling equation~\eqref{eq:pb vp buckling} with $\Lambda=-\tau$. Therefore all the solutions will have the form given
by $w(r,\theta)$ above.
\end{proof}


\begin{appendix}
\section{Eigenfunction of a clamped ball}\label{annexe:fonction propre boule}

We recall that the first normalised eigenfunction of the Dirichlet bilaplacian $S_0$ over a ball $B$ of radius $R$ is given by (see \cite[Proposition 15]{leylekian1})
$$
u_{B,0}(r)=\pm\frac{1}{\sqrt{d|B|}}\left[\frac{J_{\nu+1}(\gamma_\nu r/R)}{J_\nu(\gamma_\nu)}-\frac{I_{\nu+1}(\gamma_\nu r/R)}{I_\nu(\gamma_\nu)}\right]\left(\frac{r}{R}\right)^{-\nu},
$$
where $\nu=d/2-1$, $J_\nu$ and $I_\nu$ stand for the Bessel and modified Bessel functions of order $\nu$, and $\gamma_\nu$ is the first positive zero of $f_\nu$ defined by
$$
f_\nu(r)=\left[\frac{J_{\nu+1}}{J_\nu}(r)+\frac{I_{\nu+1}}{I_\nu}(r)\right]r^{d-1}.
$$
In this appendix, we compute the $L^2$ norm of the gradient of $u_{B,0}$, as this plays a role in Section~\ref{sec:behaviour wrt tension}.

\begin{lemme}\label{lemme:norme L2 grad u0}
Over the ball $B$ of radius $R$, the first $L^2$-normalised eigenfunction $u_{B,0}$ of $S_0$ satisfies
$$
\int_B|\nabla u_{B,0}|^2= \frac{\gamma_\nu^2}{R^2}\frac{|J_{\nu+1}(\gamma_\nu)J_{\nu-1}(\gamma_\nu)|}{J_\nu(\gamma_\nu)^2}.
$$
\end{lemme}

\begin{proof}
We have (see \cite[Appendix]{leylekian1})
$$
\partial_ru_{B,0}(r)=\frac{k}{\sqrt{d|B|}}\left[\frac{J_{\nu+1}(k r)}{J_\nu(\gamma_\nu)}+\frac{I_{\nu+1}(k r)}{I_\nu(\gamma_\nu)}\right]\left(\frac{r}{R}\right)^{-\nu},
$$
where $k:=\gamma_\nu/R$. As a result,

\begin{align}\label{eq:norme L2 derivee radiale}
\int_B (\partial_ru_{B,0})^2=\frac{k^2|\mathbb{S}^{d-1}|R^{2\nu}}{d|B|}&\left[J_\nu(\gamma_\nu)^{-2}\int_0^RJ_{\nu+1}(k r)^2r^{d-2\nu-1}dr\right.\nonumber\\
&\left.+I_\nu(\gamma_\nu)^{-2}\int_0^RI_{\nu+1}(k r)^2r^{d-2\nu-1}dr\right.\\
&\left.+2J_\nu(\gamma_\nu)^{-1}I_\nu(\gamma_\nu)^{-1}\int_0^RI_{\nu+1}(kr)J_{\nu+1}(kr)r^{d-2\nu-1}dr\right].\nonumber
\end{align}
But recall that $d-2\nu-1=1$ and that for all $\alpha\neq\beta$ in $\C$ and $\mu>-1$, by \cite[§6.521, formula~1]{gradshteyn-ryzhik}
\begin{equation*}
\int_0^1xJ_{\mu}(\alpha x)J_\mu(\beta x)dx=\frac{\beta J_{\mu-1}(\beta)J_\mu(\alpha)-\alpha J_{\mu-1}(\alpha)J_\mu(\beta)}{\alpha^2-\beta^2}.
\end{equation*}
Applying it to $\alpha=\gamma_\nu$ and $\beta=i\gamma_\nu$, this shows that
$$
\int_0^RI_{\nu+1}(kr)J_{\nu+1}(kr)r^{d-2\nu-1}dr=(R^2/2\gamma_\nu)[J_{\nu+1}(\gamma_\nu)I_\nu(\gamma_\nu)-I_{\nu+1}(\gamma_\nu)J_{\nu}(\gamma_\nu)]
$$
Then, by \cite[equation (23)]{leylekian1} (note that there is a mistake in this formula: a coefficient $2$ is missing in front of the cross product), for all $\beta\in\R$
\begin{equation}\label{eq:integrale Jmu^2}
\int_0^1 xJ_\mu(\beta x)^2dx=\frac{1}{2}\left[J_{\mu}(\beta)^2+J_{\mu-1}(\beta)^2-\frac{2\mu}{\beta} J_{\mu}(\beta)J_{\mu-1}(\beta)\right].
\end{equation}
Due to the isolation of zeros of holomorphic functions, the formula actually holds for any $\beta\in\C$. Thus we apply it to $\beta=\gamma_\nu$ and $\beta=i\gamma_\nu$, and find
$$
\int_0^RJ_{\nu+1}(k_\nu r)^2r^{d-2\nu-1}=\frac{R^2}{2}\left[J_{\nu+1}(\gamma_\nu)^2+J_\nu(\gamma_\nu)^2-\frac{2(\nu+1)}{\gamma_\nu} J_{\nu+1}(\gamma_\nu)J_\nu(\gamma_\nu)\right];
$$
$$
\int_0^RI_{\nu+1}(k_\nu r)^2r^{d-2\nu-1}=\frac{R^2}{2}\left[I_{\nu+1}(\gamma_\nu)^2-I_\nu(\gamma_\nu)^2+\frac{2(\nu+1)}{\gamma_\nu} I_{\nu+1}(\gamma_\nu)I_\nu(\gamma_\nu)\right].
$$
Consequently, and since $|\mathbb{S}^{d-1}|R^{2(\nu+1)}=d|B|$, we obtain
\begin{alignat*}{3}
\int_B (\partial_ru_{B,0})^2=&k^2\frac{|\mathbb{S}^{d-1}|R^{2(\nu+1)}}{2d|B|}&&\left[\frac{J_{\nu+1}(\gamma_\nu)^2}{J_\nu(\gamma_\nu)^{2}}+1-\frac{2(\nu+1)}{\gamma_\nu}\frac{J_{\nu+1}(\gamma_\nu)}{J_\nu(\gamma_\nu)}\right.\\
& &&\frac{I_{\nu+1}(\gamma_\nu)^2}{I_\nu(\gamma_\nu)^{2}}-1+\frac{2(\nu+1)}{\gamma_\nu}\frac{I_{\nu+1}(\gamma_\nu)}{I_\nu(\gamma_\nu)}\\
& &&\left.+\frac{2}{\gamma_\nu}\left(\frac{J_{\nu+1}(\gamma_\nu)}{J_\nu(\gamma_\nu)}-\frac{I_{\nu+1}(\gamma_\nu)}{I_\nu(\gamma_\nu)}\right)\right]\\
=& \mathrlap{k^2\left[\frac{J_{\nu+1}(\gamma_\nu)^2}{J_\nu(\gamma_\nu)^2}-\frac{2\nu}{\gamma_\nu}\frac{J_{\nu+1}(\gamma_\nu)}{J_\nu(\gamma_\nu)}\right]}.
\end{alignat*}
The last simplification comes from the fact that $f_\nu(\gamma_\nu)=0$. This is the desired formula, by virtue of the recurrence relation $J_{\nu+1}(\gamma_\nu)-2\nu J_\nu(\gamma_\nu)/\gamma_\nu=-J_{\nu-1}(\gamma_\nu)$.
\end{proof}

\section{Shape derivatives}\label{annexe:derivee de forme}

\begin{lemme}\label{lemme:derivee de forme}
Let $\Omega$ be a $C^4$ regular bounded open subset of $\R^d$ and $\tau\in\R$. Let $\Gamma$ be a simple eigenvalue of problem \eqref{eq:pb vp} and $u$ be an associated $L^2$-normalised eigenfunction. Then, for $V\in W^{2,\infty}(\R^d,\R^d)$ in some neighbourhood of $0$, there exists an eigenvalue $\gamma[V]$ with associated $L^2$-normalised eigenfunction $u[V]$ for problem \eqref{eq:pb vp} over $(id+V)\Omega$, such that $\gamma[0]=\Gamma$ and $u[0]=u$. Furthermore, $V\mapsto \gamma[V]$ and $V\mapsto u[V]\in H^1(\R^d)$ are $C^1$ in a neighbourhood of $0$. Lastly, the derivative $u'$ of $u[\cdot]$ at $0$ in the direction of a given $V$ satisfies $u'\in H_0^1\cap H^2(\Omega)$ and
$$
\partial_nu'=-\partial_n^2 u V\cdot\vec{n}.
$$
\end{lemme}

\begin{proof}
Due to the regularity assumptions, and since $\Gamma$ is simple, this can be performed thanks to the implicit function Theorem, in the wake of \cite[section~5.7]{henrot-pierre}. Indeed, define the operator
$$
\mathcal{F}:\left(
\begin{array}{ccc}
W^{2,\infty}\times H_0^2(\Omega)\times\R & \rightarrow & H^{-2}(\Omega)\times\R \\
V,v,\gamma & \mapsto &  P_V(\Delta^2-\tau\Delta-\gamma)P_V^{-1}v;\int_{\Omega}v^2Jac(id+V)
\end{array}
\right),
$$
where $P_V$ is the operation of precomposition by $(id+V)$, that is $P_V\varphi:=\varphi\circ(id+V)$. The operator $\mathcal{F}$ has the property that $\mathcal{F}(V,u\circ(id+V),\gamma)=(0,1)$ if and only if $u$ is an $L^2$-normalised eigenfunction associated with $\gamma$ over the set $(id+V)\Omega$. Moreover, at least for small $V$, the operator is smooth, and its differential with respect to $(v,\gamma)$ at the point $(0,u,\Gamma)$ is given by
$$
D\mathcal{F}(0,u,\Gamma)(v,\gamma)=\left((\Delta^2-\tau\Delta-\Gamma)v-\gamma u;2\int_\Omega uv\right).
$$
Since $\Gamma$ is assumed simple, one can prove, using Fredholm's Theorem, that the differential is an isomorphism (see \cite[Lemme 5.7.3]{henrot-pierre} for details). As a consequence, the implicit function Theorem provides smooth functions $v[V]$ and $\gamma[V]$ defined in a neighbourhood of $V=0$ such that $\mathcal{F}(V,v[V],\gamma[V])=(0,1)$. In other words, $u[V]:=P_V^{-1}v[V]$ is an $L^2$-normalised eigenfunction associated with $\gamma[V]$ over $(id+V)\Omega$. Observe that, by simplicity of $\Gamma$, we necessarily have $\gamma[0]=\Gamma$, and (up to a sign change) $u[0]=u$, for small enough $V$.

Now recall that precomposing $v[V]$ by $(id+V)^{-1}$ makes the smoothness of $u[V]$ occur in a space with reduced regularity (see \cite[Lemme 5.3.3 and Lemme 5.3.9]{henrot-pierre}). More precisely, it is the map $V\mapsto u[V]\in H^1(\R^d)$ which is $C^1$. Its derivative at $0$ in the direction of some given vector field $V$ satisfies, by the chain rule,
$$
\partial_Vu[0]+\nabla u[0]\cdot V=\partial_Vv[0]\in H_0^2(\Omega).
$$
In particular, since $\Omega$ is $C^4$, $u=u[0]$ is in $H^4(\Omega)$, hence $u':=\partial_Vu[0]$ is actually in $H^2(\Omega)$. Moreover we have
$$u'=0\text{ on }\partial\Omega,\qquad \partial_n u'=-\partial_n(V\cdot\nabla u)=-\partial_n^2uV\cdot\vec{n}\text{ on }\partial\Omega.
$$
\end{proof}

\section{The torsional two-ball problem}\label{annexe:resolution pb 2 boules torsion}

In this section, we prove Proposition~\ref{prop:pb 2 boules torsion}, which amounts to showing that the two-ball problem~\eqref{eq:pb 2 boules complet} decreases as the $B_a$ ball grows. The proof
is inspired by that of Theorem~2.3 in \cite{ashbaugh-bucur-laugesen-leylekian} and from notes of D. Bucur \cite{bucur2024}. Let $(v,w)$ be a minimising pair for $E(a,b,\tau)$ defined by~\eqref{eq:pb 2 boules complet}. Observe that $v$ satisfies the Euler-Lagrange equation
$$
\left\{
\begin{array}{rcll}
\Delta^2 v-\tau\Delta v & = & 1 &\text{in }B_a,\\
v & = & 0 &\text{on }\partial B_a,\\
\partial_n v & = & t &\text{on }\partial B_a,
\end{array}
\right.
$$
with some unknown parameter $t\in\R$. Define the functions $v_a$ and $h_a$ in the following way
\begin{equation}\label{eq:va et ha}
\left\{
\begin{array}{rcll}
\Delta^2 v_a-\tau\Delta v_a & = & 1 &\text{in }B_a,\\
v_a & = & 0 &\text{on }\partial B_a,\\
\partial_n v_a & = & 0 &\text{on }\partial B_a,
\end{array}
\right.
\qquad
\left\{
\begin{array}{rcll}
\Delta^2 h_a-\tau\Delta h_a & = & 0 &\text{in }B_a,\\
h_a & = & 0 &\text{on }\partial B_a,\\
\partial_n h_a & = & 1 &\text{on }\partial B_a.
\end{array}
\right.
\end{equation}
In this fashion, $v$ can be decomposed in the form $v=v_a+t h_a$. Moreover,
\[
\begin{array}{lll}
\dint_{B_a}(\Delta v)^2+\tau\dint_{B_a}|\nabla v|^2-2\dint_{B_a} v & = & \dint_{\partial B_a} \Delta v\partial_nv-\dint_{B_a} v\eqskip
& = & t^2\dint_{\partial B_a}\Delta h_a-2t\dint_{B_a}h_a-\dint_{B_a}v_a,
\end{array}
\]
where we used that $\int_{\partial B_a}\Delta v_a=-\int_{B_a}h_a$, which follows by combining the equations
satisfied by $h_{a}$ and $v_{a}$. A similar analysis on the ball $B_b$ yields
$$
\int_{B_b}(\Delta w)^2=s^2\int_{\partial B_b}\Delta h_b^0,
$$
where $s=\partial_nw$ is a parameter and $h_b^0$ satisfies
\begin{equation}\label{eq:vb}
\left\{
\begin{array}{rcll}
\Delta^2 h_b^0 & = & 0 &\text{in }B_b,\\
h_b^0 & = & 0 &\text{on }\partial B_b,\\
\partial_n h_b^0 & = & 1 &\text{on }\partial B_b.
\end{array}
\right.
\end{equation}
As a result, taking into account the constraint $t a^{d-1}=s b^{d-1}$, we obtain that
$$
E(a,b,\tau)=t^2\left[\int_{\partial B_a}\Delta h_a+\frac{a^{2d-2}}{b^{2d-2}}\int_{\partial B_b}\Delta h_b^0\right]-2t\int_{B_a}h_a-\int_{B_a}v_a,
$$
which is minimal for
$$
t=\frac{\int_{B_a}h_a}{\int_{\partial B_a}\Delta h_a+\frac{a^{2d-2}}{b^{2d-2}}\int_{\partial B_b}\Delta h_b^0}.
$$
Hence, we end up with the following expression for the two-ball energy:
\begin{equation}\label{eq:pb 2 boules torsion explicite partiel}
E(a,b,\tau)=-\frac{\left(\int_{B_a}h_a\right)^2}{\int_{\partial B_a}\Delta h_a+\frac{a^{2d-2}}{b^{2d-2}}\int_{\partial B_b}\Delta h_b^0}-\int_{B_a}v_a.
\end{equation}
At this point, observe that one may compute explicitly
\begin{equation}\label{eq:integrale laplacien hb0}
\int_{\partial B_b}\Delta h_b^0=d^2|B_1|b^{d-2},
\end{equation}
which simplifies the above expression. Now we need to compute the derivative of the energy with respect to $a\in (0,1)$, and subject to the constraint $a^d+b^d=1$. First, we compute the derivative of $\int_{B_a}v_a$. To do so, we resort to shape derivatives \cite[chapter 5]{henrot-pierre} to obtain
$$
\frac{d}{da}\int_{B_a}v_a=\int_{B_a}\dot{v}_a,
$$
where
$$
\left\{
\begin{array}{rcll}
\Delta^2 \dot{v}_a-\tau\Delta \dot{v}_a & = & 0 &\text{in }B_a,\\
\dot{v}_a & = & 0 &\text{on }\partial B_a,\\
\partial_n \dot{v}_a & = & -\Delta v_a &\text{on }\partial B_a,
\end{array}
\right.
$$
Multiplying by $v_a$ and integrating, one finds
\begin{equation}\label{eq:derivee va}
\frac{d}{da}\int_{B_a}v_a=\int_{\partial B_a}(\Delta v_a)^2=\frac{1}{|\partial B_a|}\left(\int_{\partial B_a} \Delta v_a\right)^2=\frac{1}{|\partial B_a|}\left(\int_{B_a} h_a\right)^2.
\end{equation}
Now it remains to compute the derivatives of $\int_{B_a}h_a$ and $\int_{\partial B_a}\Delta h_a$. Actually, note that the former can be expressed in terms of the latter. Indeed, since $\Delta h_a-\tau h_a$ is harmonic and constant on the boundary $\partial B_a$, this function is actually constant over the whole ball $B_a$. In other words, there exists $c\in\R$ such that $\Delta h_a-\tau h_a=c$. Integrating over $B_a$ on the one hand, and over $\partial B_a$ on the other hand, we obtain $|\partial B_a|-\tau\int_{B_a}h_a=c|B_a|$ and $\int_{\partial B_a} \Delta h_a=c|\partial B_a|$. Consequently,
\begin{equation}\label{eq:relation ha et laplacien ha}
\int_{B_a}h_a=\frac{a}{\tau d}\left(d^2a^{d-2}|B_1|-\int_{\partial B_a}\Delta h_a\right),
\end{equation}
and the energy turns into
$$
E(a,\tau)=-\frac{a^2}{d^2\tau^2}\cdot\frac{\left(d^2a^{d-2}|B_1|-\int_{\partial B_a}\Delta h_a\right)^2}{\int_{\partial B_a}\Delta h_a+\frac{a^{2d-2}}{b^{d}}d^2|B_1|}-\int_{B_a}v_a.
$$
To compute the derivative of $\int_{\partial B_a}\Delta h_a$, we use that
$$
\int_{\partial B_a}\Delta h_a=\int_{B_a}(\Delta h_a)^2+\tau\int_{B_a}|\nabla h_a|^2=\min_{h}\left[\int_{B_a}(\Delta h)^2+\tau\int_{B_a}|\nabla h|^2\right],
$$
where the functions $h$ evolve in the space $H_0^1\cap H^2(B_a)$ and satisfy $\partial_n h=1$. Then, the change of variables $aH(x)=h(ax)$ shows that
$$
\int_{\partial B_a}\Delta h_a=\min_H \left[a^{d-2}\int_{B_1}(\Delta H)^2+\tau a^d\int_{B_1}|\nabla H|^2\right],
$$
where the minimum is now taken among functions $H\in H_0^1\cap H^2(B_1)$ with $\partial_n H=1$, and is attained by $H_a(x):=h_a(a x)/a$. This formulation allows to differentiate with respect to $a$:
\begin{equation}\label{eq:derivee laplacien ha bord}
\begin{array}{lll}
\fr{d}{da}\dint_{\partial B_a}\Delta h_a & = & \fr{d-2}{a}a^{d-2}\dint_{B_1}(\Delta H_a)^2+\fr{\tau d}{a}a^d\dint_{B_1}|\nabla H_a|^2\eqskip
& = & \fr{d-2}{a}\dint_{B_a}(\Delta h_a)^2+\fr{\tau d}{a}\dint_{B_a}|\nabla h_a|^2\eqskip
& = & \fr{d}{a}\dint_{\partial B_a}\Delta h_a-\fr{2}{a}\dint_{B_a}(\Delta h_a)^2.
\end{array}
\end{equation}

As a consequence of equation \eqref{eq:derivee va}, equation \eqref{eq:derivee laplacien ha bord}, and of $b^d=1-a^d$, one computes
\begin{alignat*}{3}
\partial_a E(a,\tau)= & \mathrlap{\left[-\frac{2a}{d^2\tau^2}\left(d^2a^{d-2}|B_1|-\int_{\partial B_a}\Delta h_a\right)^2
\left(\int_{\partial B_a}\Delta h_a+\frac{a^{2d-2}}{b^{d}}d^2|B_1|\right) \right.}\\
& -\frac{2a^2}{d^2\tau^2} && \left(d^2(d-2)a^{d-3}|B_1|-\frac{d}{a}\int_{\partial B_a}\Delta h_a+\frac{2}{a}\int_{B_a}(\Delta h_a)^2\right) \\
& && \left(d^2a^{d-2}|B_1|-\int_{\partial B_a}\Delta h_a\right)\left(\int_{\partial B_a}\Delta h_a+\frac{a^{2d-2}}{b^{d}}d^2|B_1|\right) \\
& +\frac{a^2}{d^2\tau^2} && \left(d^2a^{d-2}|B_1|-\int_{\partial B_a}\Delta h_a\right)^2 \\
& &&\left(\frac{d}{a}\int_{\partial B_a}\Delta h_a-\frac{2}{a}\int_{B_a}(\Delta h_a)^2+2d^2(d-1)|B_1|\frac{a^{2d-3}}{b^d}+d^3|B_1|\frac{a^{3d-3}}{b^{2d}}\right) \\
& \mathrlap{\left. -\frac{1}{|\partial B_a|}\left(\int_{B_a}h_a\right)^2\left(\int_{\partial B_a}\Delta h_a+\frac{a^{2d-2}}{b^{d}}d^2|B_1|\right)^2\right]}\\
& \mathrlap{\hspace{5mm}\times\left[\int_{\partial B_a}\Delta h_a+\frac{a^{2d-2}}{b^{d}}d^2|B_1|\right]^{-2}}\\
= & \mathrlap{\left[-\frac{2}{a}\left(\int_{B_a}h_a\right)^2
\left(-\frac{\tau d}{a}\int_{B_a}h_a+a^{d-2}(1+a^d/b^d)d^2|B_1|\right) \right.}\\
& -\frac{2a}{\tau d} && \left(\frac{\tau d^2}{a^2}\int_{B_a}h_a+\frac{2\tau^2}{a}\int_{B_a}h_a^2-\frac{2\tau^2}{a^{d+1}|B_1|}\left(\int_{B_a}h_a\right)^2\right)
\left(\int_{B_a}h_a\right) \\
& && \left(-\frac{\tau d}{a}\int_{B_a}h_a+a^{d-2}(1+a^d/b^d)d^2|B_1|\right) \\
& \mathrlap{+\left(\int_{B_a}h_a\right)^2\left(-\frac{\tau d^2}{a^2}\int_{B_a}h_a-\frac{2\tau^2}{a}\int_{B_a}h_a^2+\frac{2\tau^2}{a^{d+1}|B_1|}\left(\int_{B_a}h_a\right)^2 \right.}\\
& && \left.+d^2|B_1|a^{d-3}(d\left(1+a^{d}/b^{d}\right)^2-2\left(1+a^d/b^d\right))\right) \\
& \mathrlap{\left. -\frac{1}{|\partial B_a|}\left(\int_{B_a} h_a\right)^2\left(-\frac{\tau d}{a}\int_{B_a}h_a+a^{d-2}(1+a^d/b^d)d^2|B_1|\right)^2\right]}\\
&\mathrlap{\hspace{5mm}\times\left[\int_{\partial B_a}\Delta h_a+\frac{a^{2d-2}}{b^{d}}d^2|B_1|\right]^{-2}}
\end{alignat*}
Observe that, to obtain the last equality,  we eliminated $\Delta h_a$ using relation~\eqref{eq:relation ha et laplacien ha} and
$$
\int_{B_a}(\Delta h_a)^2=d^2a^{d-2}|B_1|+\tau^2\int_{B_a}h_a^2-\frac{\tau^2}{a^d|B_1|}\left(\int_{B_a}h_a\right)^2,
$$
which comes after squaring and integrating the identity $\Delta h_a=c+\tau h_a$ over $B_a$, where $c|B_a|=|\partial B_a|-\tau\int_{B_a}h_a$. After developing and simplifying the previous expression, we find
\begin{align}\label{eq:derivee pb 2 boules}
\partial_a E(a,\tau)= & \left[\frac{\tau d(d+2)}{a^2}(1+2(1+a^d/b^d))\left(\int_{B_a}h_a\right)^3-2(d+2)d^2|B_1|a^{d-3}(1+a^d/b^d)\left(\int_{B_a} h_a\right)^2\right.\nonumber\\
& +\frac{2\tau^2}{a}\int_{B_a}h_a^2\left(\int_{B_a}h_a\right)^2-4\tau d|B_1|a^{d-2}(1+a^d/b^d)\int_{B_a}h_a\int_{B_a}h_a^2 \\
& \left.-\frac{(d+2)\tau^2}{a^{d+1}|B_1|}\left(\int_{B_a}h_a\right)^4\right] \times\left[\int_{\partial B_a}\Delta h_a+\frac{a^{2d-2}}{b^{d}}d^2|B_1|\right]^{-2}\nonumber
\end{align}

To go further, we need to compute the quantities involving $h_a$. To that end, observe that it is possible to express explicitly $h_a$ in terms of Bessel functions. Let us focus for the moment on the case $\tau>0$. In this instance one has that
\begin{equation}\label{eq:ha}
h_a(r)=
\frac{1}{\sqrt{\tau}}\left[\frac{I_\nu(r\sqrt{\tau})}{I_{\nu+1}(a\sqrt{\tau})}\left(\frac{r}{a}\right)^{-\nu}-\frac{I_\nu(a\sqrt{\tau})}{I_{\nu+1}(a\sqrt{\tau})}\right]
\end{equation}
solves \eqref{eq:va et ha}. Then, using $d-1-\nu=\nu+1$ and the formula (see \cite[section 6.561, formula~7]{gradshteyn-ryzhik})
$$
\int_0^1x^{\nu+1}I_\nu(\alpha x)dx=\alpha^{-1}I_{\nu+1}(\alpha),
$$
we may compute explicitly
\begin{equation}\label{eq:integrale ha}
\int_{B_a}h_a=a^d\frac{|B_1|}{\sqrt{\tau}}\left(\frac{d}{a\sqrt{\tau}}-\frac{I_\nu(a\sqrt{\tau})}{I_{\nu+1}(a\sqrt{\tau})}\right).
\end{equation}
Similarly, recalling also (see e.g. \cite[equation (23)]{leylekian1}, but pay attention to the factor $2$ missing in front of the cross-product there) the formula, for all $\beta\in\C$,
$$
\int_0^1xI_{\nu}(\beta x)^2dx=\frac{1}{2}[-I_{\nu+1}(\beta)^2+I_\nu(\beta)^2-\frac{2\nu}{\beta}I_{\nu+1}(\beta)I_\nu(\beta)],
$$
we obtain
$$
\int_{B_a}h_a^2=a^d\frac{|B_1|}{2\tau}\left[(d+2)\frac{I_\nu(a\sqrt{\tau})}{I_{\nu+1}(a\sqrt{\tau})}\left(\frac{I_\nu(a\sqrt{\tau})}{I_{\nu+1}(a\sqrt{\tau})}-\frac{d}{a\sqrt{\tau}}\right)-d\right].
$$
Thanks to these expressions, we shall expand the terms in \eqref{eq:derivee pb 2 boules}. In the following equation, for readability, we do not write the argument $a\sqrt{\tau}$ in the Bessel functions.
\begin{align*}
\partial_a E(a,\tau)= & \left\{d(d+2)(1+2(1+a^d/b^d))a^2\left(\frac{d^2}{a^2\tau}+\frac{I_\nu^2}{I_{\nu+1}^2}-2\frac{d I_\nu}{a\sqrt{\tau}I_{\nu+1}}\right)\right.\nonumber\\
&-\frac{2d^2(d+2)}{\sqrt{\tau}}(1+a^d/b^d)a\left(\frac{d}{a\sqrt{\tau}}-\frac{I_\nu}{I_{\nu+1}}\right)\nonumber\\
& +\sqrt{\tau}a^3\left[(d+2)\frac{I_\nu}{I_{\nu+1}}\left(\frac{I_\nu}{I_{\nu+1}}-\frac{d}{a\sqrt{\tau}}\right)-d\right]\left(\frac{d}{a\sqrt{\tau}}-\frac{I_\nu}{I_{\nu+1}}\right)\\
& -2 da^2(1+a^d/b^d)\left[(d+2)\frac{I_\nu}{I_{\nu+1}}\left(\frac{I_\nu}{I_{\nu+1}}-\frac{d}{a\sqrt{\tau}}\right)-d\right] \\
& \left.-(d+2)\sqrt{\tau}a^3\left(\frac{d^3}{a^3\tau\sqrt{\tau}}-3\frac{d^2I_\nu}{a^2\tau I_{\nu+1}}+3\frac{dI_\nu^2}{a\sqrt{\tau}I_{\nu+1}^2}-\frac{I_\nu^3}{I_{\nu+1}^3}\right)\right\} \\
& \hspace{5mm}\times\left\{\int_{\partial B_a}\Delta h_a+\frac{a^{2d-2}}{b^{d}}d^2|B_1|\right\}^{-2}|B_1|^2a^{2d-4}\int_{B_a}h_a\nonumber\\
= & \left\{d^2a^2(1+2a^d/b^d)+d\sqrt{\tau}a^3\frac{I_\nu}{I_{\nu+1}}\right\}
 \left\{\int_{\partial B_a}\Delta h_a+\frac{a^{2d-2}}{b^{d}}d^2|B_1|\right\}^{-2}|B_1|^2a^{2d-4}\int_{B_a}h_a.\nonumber\\
\end{align*}

Since the Bessel functions $I_\nu$ and $I_{\nu+1}$ are non-negative, the sign of the derivative is given by the sign of $\int_{B_a}h_a$. But $h_a$ is negative (recall that $(I_{\nu}(r)r^{-\nu})'=I_{\nu+1}(r)r^{-\nu}\geq0$). Therefore, $E(\cdot,\tau)$ is decreasing over $(0,1)$ for all $\tau>0$. To conclude, notice that combining \eqref{eq:pb 2 boules torsion explicite partiel}, \eqref{eq:integrale laplacien hb0}, \eqref{eq:relation ha et laplacien ha}, \eqref{eq:integrale ha}, and \eqref{eq:rigidité torsionnelle de la boule}, we obtain the explicit expression \eqref{eq:pb 2 boules torsion explicite} for the two-ball energy.

For $\tau=0$, observing that $E(a,\tau)$ depends smoothly on $\tau\geq0$ for $a\in(0,1)$ (see for instance~\eqref{eq:pb 2 boules torsion explicite partiel}), it is enough to take the limit $\tau\to0$ in the above equation to see that $\partial_aE(a,0)$ is negative as well for $a\in(0,1)$. Lastly, taking $\tau\to0$ in~\eqref{eq:pb 2 boules torsion explicite} and using the asymptotic expansion
\begin{equation}\label{eq:expansion ratio Bessel}
x\fr{I_\nu(x)}{I_{\nu+1}(x)}\underset{x\to0}{=}d+\fr{x^2}{d+2}-\frac{x^4}{(d+2)^2(d+4)}+o(x^4)
\end{equation}
gives the expression~\eqref{eq:pb 2 boules torsion explicite tau=0}.
\end{appendix}

\begin{acknowledgements}
We would like to thank D. Bucur and M. S. Ashbaugh for the many discussions while writing this article. We also want to thank D. Buoso and R.S. Laugesen for their comments and suggestions on the manuscript. This work was partially supported by the Fundação para a Ciência e a Tecnologia, I.P. (Portugal), through project {\small UID/00208/2023}.
\end{acknowledgements}



\printbibliography



\end{document}